\documentclass[12pt]{amsart}
\usepackage{txfonts}
\usepackage{amscd,amssymb,mathabx,mathtools}
\usepackage[arrow,matrix,graph,frame,poly,arc,tips]{xy}
\usepackage{graphicx,calrsfs}
\usepackage{tikz}
\usetikzlibrary{decorations.markings}
\usetikzlibrary{shapes}
\usetikzlibrary{backgrounds}
\usetikzlibrary{calc}
\usepackage{mdwlist}
\usepackage{xfrac}

\usepackage{xcolor}
\usepackage{multirow}
\usepackage{enumerate}
\usepackage{verbatim}
\usepackage{comment}
\usepackage{todonotes}

\newcommand\on[1]{\operatorname{#1}}

    \newcommand\ba{\begin{align*}}
    \newcommand\ea{\end{align*}}
    \newcommand\be{\begin{enumerate}}
    \newcommand\ee{\end{enumerate}}
    \newcommand\bp{\begin{proof}}
    \newcommand\ep{\end{proof}}
    \newcommand\bpp{\begin{prop}}
    \newcommand\epp{\end{prop}}
    \newcommand\bpb{\begin{prob}}
    \newcommand\epb{\end{prob}}
    \newcommand\bd{\begin{defn}}
    \newcommand\ed{\end{defn}}
    \newcommand\bh{\begin{hint}}
    \newcommand\eh{\end{hint}}

    \newcommand\N{\mathbb{N}}

    \newcommand\Q{\mathbb{Q}}
    
    \newcommand\Z{\mathbb{Z}}

    \newcommand\GG{\mathcal{G}}
    \newcommand\DD{\mathcal{D}}
    \newcommand\HH{\mathcal{H}}

\newcommand\UU{\mathcal{U}}

    \newcommand{\R}[0]{\mathbb{R}}

    \newcommand\Str{\operatorname{Str}}
     \newcommand\Ad{\operatorname{Ad}}
    \newcommand\Add{\operatorname{Add}}
    \newcommand\mesh{\operatorname{mesh}}
    
    \newcommand\Mult{\operatorname{Mult}}
     \newcommand\id{\operatorname{id}}

    \newcommand\supp{\operatorname{supp}}

    \newcommand\diam{\operatorname{diam}}

    \newcommand\cc{\mathrm{cc}}

    \DeclareMathOperator\Homeo{Homeo}

    \DeclareMathOperator\Fix{Fix}
    
    \DeclareMathOperator\Aut{Aut}

    \DeclareMathOperator\Diff{Diff}
    \DeclareMathOperator\Exp{Exp}

    \DeclareMathOperator\ro{RO}
    
    \DeclareMathOperator\PL{PL}

    \DeclareMathOperator\cl{cl}

    \DeclareMathOperator\Th{{Th}}

    \newcommand{\arxiv}[1]
    {\texttt{\href{http://arxiv.org/abs/#1}{arXiv:#1}}}
    \newcommand{\doi}[1]
    {\texttt{\href{http://dx.doi.org/#1}{doi:#1}}}
    \renewcommand{\MR}[1]
    {\href{http://www.ams.org/mathscinet-getitem?mr=#1}{MR#1}}

    \def\thetitle{{Locally approximating groups of homeomorphisms of manifolds}}
    \def\theauthors{{Thomas Koberda, J.~de la Nuez Gonz\'alez}}
    \usepackage{hyperref}
    \hypersetup{%
        hidelinks,
        plainpages=true,
        pdftitle={\thetitle},
        pdfauthor={\theauthors}
    }

    \newtheorem{thm}{Theorem}[section]
    \newtheorem{lem}[thm]{Lemma}
    
    \newtheorem{cor}[thm]{Corollary}
    \newtheorem{prop}[thm]{Proposition}

    \newtheorem*{claim*}{Claim}

    \theoremstyle{remark}
    
    \newtheorem{rem}[thm]{Remark}

    \theoremstyle{definition}
    \newtheorem{defn}[thm]{Definition}
    \newtheorem{prob}{Problem}[section]

\usepackage{mathrsfs}
\providecommand{\cl}{\operatorname{cl}}
\providecommand{\Int}{\operatorname{int}}
\providecommand{\Loc}{\operatorname{Loc}}
\providecommand{\Buf}{\operatorname{Buf}_2}
\providecommand{\FinFoot}{\operatorname{FinFoot}}
\providecommand{\SameFoot}{\operatorname{SameFoot}}
\providecommand{\Card}{\operatorname{Card}}
\providecommand{\Num}{\operatorname{Num}}
\providecommand{\Bij}{\operatorname{Bij}}
\providecommand{\Inj}{\operatorname{Inj}}
\providecommand{\Carry}{\operatorname{Carry}}
\providecommand{\FiberCard}{\operatorname{FiberCard}}
\providecommand{\Path}{\operatorname{Path}}
\providecommand{\Cycle}{\operatorname{Cycle}}
\providecommand{\At}{\operatorname{At}}
\providecommand{\Forward}{\operatorname{Forward}}
\providecommand{\Exp}{\operatorname{Exp}}
\providecommand{\Eval}{\operatorname{Eval}}
\providecommand{\Member}{\operatorname{Member}}
\providecommand{\Slice}{\operatorname{Slice}}
\providecommand{\Dec}{\operatorname{Dec}}
\providecommand{\Gen}{\operatorname{Gen}}
\providecommand{\suppv}{\operatorname{supp}_{K}}
\providecommand{\St}{\operatorname{St}}
\providecommand{\sd}{\operatorname{sd}}
\providecommand{\Ftwo}{\mathbb F_2}
\providecommand{\goodclass}{\mathscr A}
\providecommand{\Th}{\operatorname{Th}}
\providecommand{\Aut}{\operatorname{Aut}}
\providecommand{\Fix}{\operatorname{Fix}}
\providecommand{\LT}{\mathrm{LT}}

\providecommand{\Pow}{\operatorname{Pow}}

\begin{document}
\title\thetitle
\date{October 7, 2026}
\author[T. Koberda]{Thomas Koberda}
\address{Department of Mathematics, University of Virginia, Charlottesville, VA 22904-4137, USA}
\email{thomas.koberda@gmail.com}
\urladdr{https://sites.google.com/view/koberdat}
\author[J. de la Nuez Gonz\'alez]{J. de la Nuez Gonz\'alez}
\address{School of Mathematics, Korea Institute for Advanced Study (KIAS), Seoul, 02455, Korea}
\email{jnuezgonzalez@gmail.com}
\begin{abstract}
Let $M$ be a compact, connected manifold of positive dimension, and let
$\GG\leq\Homeo(M)$ be locally approximating. If $\dim M\geq2$, we construct a
uniform parameter-free interpretation of standard first-order arithmetic in
$\GG$. The same construction applies in dimension one under a component-selection
hypothesis, which holds for the usual piecewise-linear and Thompson-group
examples, as well as for full homeomorphism and diffeomorphism groups. We also define
finite word evaluation and membership in finitely generated subgroups, and derive
consequences for finite generation, prime models, and quasi-finite axiomatizability.
In dimension one, we construct
countable recursively saturated locally approximating groups of homeomorphisms of the interval and circle
which do not interpret standard arithmetic, even with parameters. Finally, we prove an action
rigidity result: if $M$ is closed and admits a
combinatorial triangulation, $\GG\leq\Homeo(M)$ is locally approximating,
and $H\equiv\GG$ admits a locally approximating action on a compact connected
topological manifold $N$, then $N$ is homeomorphic to $M$.
\end{abstract}
\maketitle
\setcounter{tocdepth}{1}
\tableofcontents

\section{Introduction}\label{sec:intro}

Let $M$ be a compact, connected topological manifold of positive dimension,
possibly with boundary. Write $\Homeo(M)$ for the group of homeomorphisms of $M$,
and write $\Homeo_0(M)$ and $\Homeo_c(M\setminus\partial M)$ for its identity
component and for the group of compactly supported homeomorphisms of the
interior, respectively. All manifolds in this paper are Hausdorff and second
countable. A closed manifold is compact and has empty boundary.
We will investigate the first-order theory (in the language of abstract groups)
of certain, usually countable, groups of
homeomorphisms of $M$. We use standard model-theoretic terminology; see
\cite{marker,tent-ziegler}.

For an open subset $U\subseteq M$, write $\GG[U]=\{g\in\GG:g|_{M\setminus U}=\mathrm{id}\}$ for its
rigid stabilizer. We say that a subgroup $\GG\leq\Homeo(M)$ is \emph{locally approximating}
if for every open $U$ whose closure is compact and contained in one
interior coordinate chart, the group $\GG[U]$ is dense in
$\Homeo(M)[U]$ in the compact-open topology. If $\partial M\ne\varnothing$, we
allow either of the following alternatives:
\begin{enumerate}
\item $\GG\leq\Homeo_c(M\setminus\partial M)$;
\item the analogous density condition holds for chart-local open sets also in
boundary charts.
\end{enumerate}

The purpose of this paper is to relate the algebra and first-order theory of
locally approximating groups to the topology of the manifolds on which they act, thus
continuing the program of \cite{dlNKK22,dlNK23}. In particular, we will show that locally approximating groups of
homeomorphisms interpret first-order arithmetic whenever the dimension of $M$ is at least two, and
that they exhibit strong action rigidity. We thus develop a first-order version of Rubin's Theory~\cite{Rubin1989,Rubin1996}.

\subsection{Arithmetic and membership}
For $g\in\GG\leq \Homeo(M)$, put
\[
 \supp^e(g)=\Int\cl\{x\in M:g(x)\ne x\},\qquad
 \DD=\{\supp^e(g):g\in\GG\}.
\]
For $A,V\in\DD$ with $A\subseteq V$, let $\cc_V(A)$ denote the family of components of
$V$ meeting $A$. We call the action of $\GG$ \emph{component-selecting} if for every
$V\in\DD$ and every subfamily $\mathcal C\subseteq\cc(V)$ there is an
$A\in\DD$ with $A\subseteq V$ such that $\cc_V(A)=\mathcal C$. Observe that this is a hypothesis on
the action, and not an
extra symbol in the group language.

Let $\goodclass$ denote the class of locally approximating actions on
compact connected manifolds $M$ of positive dimension satisfying
\begin{equation}\label{eq:class}
 \dim M\geq2\quad\text{or}\quad\text{the action is component-selecting}.
\end{equation}
In particular, every locally approximating action in dimension at least two
belongs to $\goodclass$.

\begin{thm}\label{thm:arith}
For every action in $\goodclass$, the group $\GG$ admits a parameter-free
interpretation of
\[
 (\N,0,1,+,\times,\leq).
\]
The interpretation is uniform in $\GG$ and $M$, including the dimension of $M$.
Its underlying number structure is the standard one.
\end{thm}

An \emph{interpretation} here is a definable subset of a fixed finite Cartesian power,
modulo a definable equivalence relation, with definable operations and relations
on the quotient. In Theorem~\ref{thm:arith}, the standardness assertion is part of the theorem.
This means that the interpreted structure is isomorphic to the standard
natural numbers with their usual operations and order, and not to just a
(possibly nonstandard) model of Peano arithmetic or of true arithmetic.

\begin{thm}\label{thm:member}
For actions in $\goodclass$ there is a uniform parameter-free relation
\[
 \Exp(f,n;U,V)\quad\Longleftrightarrow\quad f^n(U)=V
 \qquad(f\in\GG,\ n\in\N,\ U,V\in\DD).
\]
For each fixed $r$, the evaluation map
\[
 (w,g_1,\ldots,g_r)\longmapsto w(g_1,\ldots,g_r)
\]
is uniformly definable, where $w$ ranges over arithmetically coded finite group
words. Consequently there is a parameter-free formula
$\Member_r(h,g_1,\ldots,g_r)$ which holds exactly when
$h\in\langle g_1,\ldots,g_r\rangle$.
\end{thm}

Uniformity in the second assertion is for each fixed arity $r$, and in particular it does not
quantify over tuples of group elements of variable length. The word variable
itself is a single element of the interpreted arithmetic sort. Theorem~\ref{thm:member} has the following
group-theoretic consequence, which in particular applies to elementary subgroups of the full
homeomorphism group of a positive-dimensional manifold.

\begin{thm}\label{thm:infgen}
Let $\GG$ admit an action in $\goodclass$, and let $H\equiv\GG$. If $\GG$ is
not finitely generated, then $H$ is not finitely generated.
\end{thm}

In Theorem~\ref{thm:infgen}, the group $H$ need not act on a manifold;
it is assumed only to be abstractly elementarily equivalent to $\GG$.
The class $\goodclass$ includes the usual actions of Thompson's groups
$F$ and $T$, locally approximating groups of finite-piece PL homeomorphisms
of compact one-manifolds, full homeomorphism groups, and full
diffeomorphism groups of compact one-manifolds. Thus
Theorem~\ref{thm:infgen} applies to these examples.

Rubin announced that the complete theory of every locally moving group
on a regular Hausdorff space is undecidable
\cite{Rubin2013announcement,Rubin2014announcement}.
Here local movement means that $G[U]$ is nontrivial for every nonempty
open set $U$ in the underlying space. His announced theorem therefore
applies to all locally approximating groups considered here, without
component selection. For actions in $\goodclass$, our interpretation of
standard arithmetic gives an independent proof of undecidability:
\begin{cor}
Let $G$ admit an action in $\goodclass$. Then the full theory of $G$ is undecidable.
\end{cor}

\subsection{Prime models and quasi-finite axiomatizability}
Recall that a countable structure is a prime model of its complete theory if it
embeds elementarily in every model of that theory. A finitely generated group
admitting an action in $\goodclass$ will be called \emph{interpretably presented} when, for some
finite generating tuple, its trivial words are specified by a parameter-free
formula on the interpreted word sort. The formula is allowed to be a group
formula; it need not be purely arithmetic. See
Definition~\ref{def:interpretable-presentation} for a more precise formulation.

\begin{thm}\label{thm:prime-intro}
Every finitely generated, interpretably presented group admitting an action in
$\goodclass$ is a prime model of its theory. If such a group is also
parameter-free interpretable in standard arithmetic, then it is bi-interpretable
with arithmetic with parameters and is quasi-finitely axiomatizable.
\end{thm}

We use the standard definition of \emph{quasi-finitely axiomatizable} (QFA): an
infinite finitely generated group $G$ is QFA if there is a \emph{single first-order
sentence} $\varphi$ such that $G\models\varphi$ and every finitely generated
group satisfying $\varphi$ is isomorphic to $G$.

A group $G$ is an \emph{elementary retract} of a group $H$ if there is an elementary embedding $i\colon G
\longrightarrow H$ and a retraction $r\colon H\longrightarrow G$ such that $r\circ i$ is the identity.

\begin{cor}\label{cor:retract-intro}
Let $G$ satisfy the primality hypotheses of Theorem~\ref{thm:prime-intro}, and
let $H\equiv G$ be finitely presented. Then $G$ is an elementary retract of
$H$. The retraction is an isomorphism if either $G$ is uniformly simple, or
$G$ is nonabelian and centerless and its derived subgroup is uniformly simple
and has bounded commutator width.
\end{cor}

These results recover the applicable conclusions for $F$ and $T$; Lasserre's
bi-interpretability theorem gives their QFA property directly
\cite{lasserre}.

\subsection{One-dimensional phenomena}
Group actions in dimension one are usually much easier to understand, though can still be quite complicated in general; see~\cite{KL2020,KKL2019,KK2021book,KK2020}; even the free product structure of various groups of homeomorphisms and diffeomorphisms is quite subtle: see~\cite{KK2018,Kob-PP}. For locally
approximating groups of homeomorphisms of the interval or of the circle, it is easy to produce many
lamplighter subgroups of the form $\Z\wr\Z$, which interpret arithmetic; see~\cite{AM2009}, for
instance. Nevertheless, these lamplighter groups are generally not definable or even interpretable,
and the one-dimensional analogue of Theorem~\ref{thm:arith} without
the component-selection hypothesis is false.

\begin{thm}\label{thm:negative-intro}
There is a countable group $G\leq\Homeo_+([0,1])$ that is locally
approximating, whose elements are all compactly supported in $(0,1)$, and
which does not interpret
standard arithmetic, even with parameters. Moreover, no
infinite finitely generated subgroup of $G$ is definable with parameters. The same conclusion holds for the circle
as well.
\end{thm}

\subsection{Action rigidity}
The second major theme in this paper is action rigidity. An action of $G$ on $X$ is
\emph{locally dense} if, for every open $U\subseteq X$ and every $p\in U$,
we have $p\in\Int\cl(G[U]\cdot p)$. Rubin's theorem
\cite{Rubin1989,Rubin1996} says that if $G$ acts
faithfully and locally densely on two locally compact Hausdorff spaces
$X$ and $Y$ without isolated points, then there is a $G$-equivariant
homeomorphism $X\to Y$. Positive-dimensional manifolds have no isolated
points.

Here, we prove a first-order version of Rubin's theorem for locally approximating groups of
homeomorphisms acting on closed manifolds; roughly, if $G$ and $H$ are locally approximating on
closed manifolds $M$ and $N$ and elementarily equivalent, then $M$ and $N$ are necessarily
homeomorphic. The crucial difference between Rubin's setting and ours is that there need not be any
map comparing $G$ and $H$.

\begin{thm}\label{thm:dimension-intro}
For every $n\geq1$ there is a group sentence $\mathrm{dim}_n$ such that a
locally approximating action on a compact connected manifold of positive
dimension satisfies $\mathrm{dim}_n$ exactly when that manifold has
dimension $n$.
\end{thm}

The analogue of Rubin's theorem is the following:

\begin{thm}\label{thm:rigidity-intro}
Let $M$ be a closed connected manifold of positive dimension admitting a
combinatorial triangulation. There is a group sentence
$\Theta_M$ such that, for every compact connected topological manifold $N$
and every locally approximating $\HH\leq\Homeo(N)$,
\[
 \HH\models\Theta_M\quad\Longleftrightarrow\quad N\cong M.
\]
In particular, if $\GG\leq\Homeo(M)$ and $\HH\leq\Homeo(N)$ are locally
approximating and $\GG\equiv\HH$, then $M$ and $N$ are homeomorphic.
\end{thm}

No triangulation, smoothness, or orientability assumption is imposed on $N$.
A \emph{topological triangulation} of $M$ is a homeomorphism $|K|\to M$
from the realization of a locally finite simplicial complex. For a closed
$n$-manifold it is
\emph{combinatorial} if the link of every $k$-simplex, for $0\leq k<n$,
is PL homeomorphic to $S^{n-k-1}$; the link of an $n$-simplex is empty.
Such a triangulation gives a compatible PL structure
\cite{RourkeSanderson}. An arbitrary topological triangulation need not
satisfy these link conditions. Our hypothesis is the existence of a
combinatorial triangulation, not merely of a topological one; the acting
homeomorphisms need not preserve a PL structure.
The theorem includes all closed connected topological surfaces and
three-manifolds, without orientability restrictions, by the classical
triangulation theorems in these dimensions \cite{Moise1952,Moise1977}.

\subsection{Organization of the paper}

Section~\ref{sec:ro} develops the support and component language.
Section~\ref{sec:arith} interprets arithmetic and exponentiation.
Section~\ref{sec:member} constructs word evaluation, derives its algebraic
consequences, treats primality and QFA, and gives the one-dimensional
counterexamples. Section~\ref{sec:rigidity} establishes dimension rigidity and
action rigidity.

\section{Dense subsets of the Boolean algebra of regular open sets}\label{sec:ro}

\subsection{Regular open sets and expressibility}
An open subset $U$ of a space $X$ is \emph{regular open} if
$U=\Int\cl U$. The regular open sets form a Boolean algebra $\ro(X)$, with
\[
 U\wedge V=U\cap V,\qquad U\oplus V=\Int\cl(U\cup V),\qquad
 U^\perp=X\setminus\cl U.
\]
The partial order is inclusion. A collection $\DD\subseteq\ro(X)$ is dense
in the Boolean-algebra sense if every nonempty regular open set contains a
nonempty member of $\DD$, though $\DD$ may not itself be a
Boolean subalgebra.

Recall that a group of homeomorphisms of a Hausdorff topological space
is \emph{locally moving} if its rigid stabilizer of every nonempty
open set is nontrivial. Local approximation implies local movement. Here $\dim M>0$ ensures
that each nonempty open set contains an interior ball admitting a nontrivial
compactly supported homeomorphism. We use the
following form of Rubin's expressibility theorem
\cite{Rubin1989,Rubin1996,KK2021book}.

\begin{thm}[Rubin expressibility]\label{thm:rubin-express}
For a locally moving group $\GG$ acting on a manifold, the extended supports
of group elements form a uniformly parameter-free interpretable sort $\DD$.
Inclusion, disjointness, the rigid stabilizer $\GG[U]$ for $U\in\DD$, and the
action $(g,U)\mapsto g(U)$ are uniformly definable. Boolean relations among a
fixed finite list of regular-open expressions in elements of $\DD$ are
first-order expressible.
\end{thm}

In particular, equality of extended supports gives the quotient defining
$\DD$. Throughout, a variable in this sort is implemented by a group variable
modulo that definable equality. Boolean expressions are allowed in predicates,
but a quantifier $\exists U\in\DD$ always requires an actual support witness.
For example, $A\subseteq B\oplus C$ is a relation and does not assert
$B\oplus C\in\DD$.

We will repeatedly use the following elementary consequence of Boolean density.
If $O$ and $R$ are regular open and every $X\in\DD$ with $X\subseteq O$
satisfies $X\subseteq R$, then $O\subseteq R$. Otherwise the nonzero Boolean
difference $O\cap R^\perp$ would contain such an $X$ disjoint from $R$.
This allows us to avoid assuming that the members of
$\DD$ form a neighborhood basis at every individual point.

If $g_1,\ldots,g_k$ have pairwise disjoint extended supports, the extended
support of their product is their Boolean join. If the closures of those
extended supports are pairwise disjoint, that join is their literal union and
no components merge.

\subsection{Ball containment}
A \emph{collared closed ball} in the interior of an $n$-manifold is the image
of the closed unit ball under an embedding extending to a neighborhood of that
ball in $\R^n$. We write $A\ll B$ if the entire set $A$ lies in a collared closed ball contained in the open set
$B$. In particular, $\cl A\subseteq B$ and $\cl A$ is compact in the interior
of $M$. We put $\varnothing\ll B$ by convention. This is stronger than only
requiring each component of $A$ to have a compact enlargement inside $B$.

\begin{lem}\label{lem:compact}
For nonempty $A,B\in\DD$, the relation $A\ll B$ is expressed by
\begin{equation}\label{eq:compact-formula}
 \exists T\in\DD\ \bigl[\varnothing\ne T\subseteq B\ \wedge\
 \forall R\in\DD\,
 (\varnothing\ne R\subseteq T\Rightarrow
       \exists g\in\GG[B]\ g(A)\subseteq R)\bigr].
\end{equation}
The formula is uniform in the group and does not depend on the underlying manifold.
\end{lem}
\begin{proof}
If the formula holds, choose $R$ in a small collared closed ball $K$ in $T$.
A witness $g$ has $A\subseteq g^{-1}(K)\subseteq B$, since $g$ is supported in
$B$. The inverse image is again a collared closed interior ball.

Conversely, suppose $A\subseteq K\subseteq B$ for a collared closed ball
$K$. A slightly larger collared ball lies in $B$. Choose a nonempty support
$T$ in the interior of $K$. A homeomorphism supported in the larger ball can
compress $K$ into a small ball in any nonempty $R\subseteq T$. The condition
that the compact set $K$ have image in the open set $R$ is compact-open open.
Local approximation in the larger ball supplies a group element with that
property, giving~\eqref{eq:compact-formula}.
\end{proof}

\begin{lem}[Support envelopes]\label{lem:large-subset}
Let $K$ be a compact subset of an interior open ball $O$. There exist
$X,Y\in\DD$ such that
\[
 K\subseteq X,\qquad X\ll Y,\qquad \cl Y\subseteq O.
\]
\end{lem}
\begin{proof}
Choose nested collared closed balls
$B_0\subset\Int B_1\subset\cdots\subset\Int B_4\subset O$, with
$K\subset\Int B_0$. A homeomorphism supported in $\Int B_2$ can move the
compact ball $B_0$ into a small ball disjoint from $B_0$. Approximate it by an
element of $\GG[\Int B_2]$, retaining that displacement. Its extended support
$X$ contains $B_0$ and has closure in $B_2$. Similarly, an element supported
in $\Int B_4$ can displace $B_3$; let $Y$ be its extended support. Then $Y$
contains $B_3$ and $B_2$ witnesses $X\ll Y$.
\end{proof}

We also use $\ll$ when its arguments are fixed regular-open Boolean
expressions. This is legitimate, since for nonempty expressions $P,Q$, the formula
\begin{equation}\label{eq:boolean-compact}
 \exists X,Y\in\DD\quad P\subseteq X\ll Y\subseteq Q
\end{equation}
is equivalent to whole-set ball containment of $P$ in $Q$.

Define
\[
 \Loc(A)\quad\Longleftrightarrow\quad
 A=\varnothing\ \vee\ \exists T\in\DD\ (A\ll T).
\]
The choice of name for the predicate is descriptive, since it suggests that $A$ is ``localized".
For nonempty $A$, this is equivalent to containment in one collared closed
interior ball. The implication to this geometric property is immediate, and
the converse follows from the support envelope lemma. All references we will make to a local support
below use this predicate, rather than an unnamed coordinate chart as a
parameter.

\subsection{Components and finite selections}
The components of a regular open set in a manifold are themselves regular
open. Indeed, regularizing one component cannot enter a different component,
which is an open set disjoint from the first; the regularization is contained
in the original regular open set. Components meeting the boundary are still
determined by their interior points. Paths in an open component can be pushed
into the interior in the boundary charts.

\begin{prop}\label{prop:interp-conn}
The group uniformly interprets the sort of components of members of $\DD$,
with incidence, containment, and equality of components, as well as the group action.
\end{prop}
\begin{proof}
For $V\in\DD$, use nonempty $X,Y\in\DD$ satisfying $X,Y\ll V$ to represent
components. Such an $X$ lies in one component of $V$. Put
\begin{equation}\label{eq:component-equivalence}
 X E_V Y\quad\Longleftrightarrow\quad
 \exists T\in\DD\ (T\ll V\ \wedge\ T\cap X\ne\varnothing
                       \ \wedge\ T\cap Y\ne\varnothing).
\end{equation}
If this holds, the ball containing $T$ connects the components represented by $X$ and $Y$, so
they coincide. Conversely, choose interior points of $X$ and $Y$ in their
common component. A tame arc joining the points, constructed through finitely
many coordinate neighborhoods, has a collared ball neighborhood in that
component. Choose a compact neighborhood of the arc inside that collared ball.
Lemma~\ref{lem:large-subset} supplies a support envelope
$T\in\DD$ contained in the ball and containing this compact neighborhood.
Then $T\ll V$, and $T$ meets both $X$ and $Y$, as required
in~\eqref{eq:component-equivalence}. Every component is represented by such an $X$. Thus
$E_V$ identifies precisely the representatives of the same component.

For an arbitrary $A\in\DD$, incidence with the component represented by
$(X,V)$ is expressed by the existence of a nonempty $T\ll V$ with
$T E_V X$ and $T\cap A\ne\varnothing$. To say that $A$ is contained in the
component represented by $X$, require $A\subseteq V$ and that every
representative meeting $A$ be $E_V$-equivalent to $X$. Inclusion of one component in another is
then tested on all support subsets of the first. Regular openness and Boolean density show that we
may test the density of $A$ in a component of $V$, and conclude equality with a component of $V$.
The claim about the group action is immediate from Rubin's expressibility theorem.
\end{proof}

Consequently, statements such as
$C\in\cc(V)$, $C\cap A\ne\varnothing$, $C\subseteq D$, and $g(C)=D$ will be
fixed abbreviations for expressible first-order properties. In particular, they do not quantify over
the power set of components of
$V$. For $A,W\subseteq V$, write
\[
 \SameFoot(A,W;V)\quad\Longleftrightarrow\quad
 \forall C\in\cc(V)\ (C\cap A\ne\varnothing\leftrightarrow
                          C\cap W\ne\varnothing).
\]

\begin{lem}[Finite selection]\label{lem:finite-selection}
Let $C_1,\ldots,C_k$ be distinct components of a regular open set $V$,
each meeting a regular open set $W$. There is $A\in\DD$ such that
$\cl A$ is a compact subset of $V\cap W\cap\Int M$ and
$\cc_V(A)=\{C_1,\ldots,C_k\}$.
\end{lem}
\begin{proof}
Choose interior open balls $B_i$ with pairwise disjoint compact closures and
$\cl B_i\subseteq C_i\cap W$. Choose a nontrivial group element supported
in each $B_i$, and let $A$ be the extended support of their product.
Then \[\cl A\subseteq\bigcup_i\cl B_i\subseteq V\cap W\cap\Int M,\]
and $A$ meets exactly the prescribed components, where for $k=0$ we use the
identity.
\end{proof}

The choices may thus be made inside any regular open set meeting the
prescribed components. We do not assert $A\ll W$: when $W$ is disconnected,
one ball in $W$ need not contain all the chosen pieces.

\begin{lem}[Finite transport]\label{lem:finite-transport}
Let $U,V\in\DD$, and choose distinct components $U_1,\ldots,U_k$ of $U$
and distinct components $V_1,\ldots,V_k$ of $V$. If $\dim M\geq2$, there
are $A\in\DD$ and $h\in\GG$ such that $\cl A$ is a compact subset of $U$,
$\cc_U(A)=\{U_1,\ldots,U_k\}$, and $h(A\cap U_i)\subseteq V_i$ for every
$i$. In dimension one, the same conclusion holds when the prescribed
correspondence preserves linear order on an interval or cyclic order on
a circle. Moreover, in every positive dimension, one element of $\GG$
can carry any finite subset of $\Int M$ into any prescribed nonempty
open subset of $\Int M$.
\end{lem}

This lemma is a weak version of ball--transitivity, which is used in~\cite{dlNKK22}. Connectedness of $M$ is used crucially here.
\begin{proof}[Proof of Lemma~\ref{lem:finite-transport}]
Choose interior points $x_i\in U_i$ and $y_i\in V_i$, with all target
points distinct from the source points. In dimension at least two,
connectedness of $\Int M$ allows the points to be moved successively
along paths avoiding the other, currently fixed points. Covering each
path by finitely many interior coordinate balls realizes these moves by
a finite product of chart-supported homeomorphisms. In dimension one,
order-compatible configurations can be joined by a path without
collisions; subdividing this path into short steps gives the same
finite-product construction in disjoint small interval neighborhoods at
each step. Thus in either case a finite product $h_0$ of chart-supported
homeomorphisms satisfies $h_0(x_i)=y_i$.

Approximate each factor using local approximation. Continuity of finite
composition allows the resulting $h\in\GG$ to retain the open conditions
$h(x_i)\in V_i$. Choose small balls $B_i$ about $x_i$, with pairwise
disjoint compact closures in $U_i\cap h^{-1}(V_i)$. Apply
Lemma~\ref{lem:finite-selection} with $W=\bigcup_iB_i$ to obtain $A$.
Then $A\cap U_i\subseteq B_i$, giving the required inclusions. For $k=0$,
take $A=\varnothing$ and $h=1$.

For the final assertion, list the distinct points of the given finite
set and choose target points in a small interval or ball inside the
prescribed open set, respecting the order when $\dim M=1$. The same
finite-product approximation sends all source points into that set.
\end{proof}

\begin{lem}\label{lem:finite-ball}
Every finite subset of the interior of a connected manifold lies in an
interior collared ball.
\end{lem}
\begin{proof}
Choose an interior collared closed ball $B$ and use the finite-point
assertion of Lemma~\ref{lem:finite-transport} to send the finite set into
$\Int B$. Its inverse image under the resulting homeomorphism is the
required collared ball.
\end{proof}

\subsection{Component selection and natural examples}\label{subsec:selection}
\begin{prop}\label{prop:component-selection}
Let $\GG$ be locally moving. Each of the following implies component selection
on the elements of $\DD$.
\begin{enumerate}
\item Every member of $\DD$ has finitely many components.
\item Every union of components of every element of $\DD$ is the extended support of
an element of $\GG$.
\end{enumerate}
In particular, every locally approximating subgroup of the PL
homeomorphism group of a compact interval or circle is component-selecting, and the
standard actions of $F$ and $T$ satisfy this condition.
\end{prop}
\begin{proof}
In the first case every selected subfamily is finite, and
Lemma~\ref{lem:finite-selection} supplies $A\in\DD$ meeting precisely
that subfamily. In the second case the
union of any chosen subfamily of components of $V$ is regular open. Its
regularization is contained in $V$ and cannot enter an omitted open component,
so it equals that union.

That the standard actions of $F$ and $T$ are locally approximating and satisfy the first condition is a standard
fact; see \cite{CFP1996,AM2009}.
\end{proof}

Every subgroup of $\Homeo(M)$ containing $\Homeo_0(M)$ is
component-selecting: every regular open set lies in $\DD$
\cite[Proposition~2.7]{dlNKK22}. The same support-realization property
holds for groups of diffeomorphisms containing $\Diff_0^\infty(M)$ on a
smooth manifold; see \cite[proof of Corollary~2.5]{dlNKK25}.
This proves component selection, not local approximation. For the full
one-dimensional diffeomorphism groups, local approximation follows from
smooth increasing interpolation on the finitely many intervals needed
to meet a compact-open neighborhood, as in the proof of
Lemma~\ref{lem:ordered-density}. In higher dimensions, their inclusion in
$\goodclass$ is conditional on the separate local-approximation hypothesis.

\section{Interpreting arithmetic in locally approximating groups of homeomorphisms}\label{sec:arith}

In this section, we fix an ambient action belonging to $\goodclass$. The manifold $M$ and coordinate
charts inside of it will be implicit throughout. We first construct
a finiteness certificate in dimensions at least two. We then retain finite
strings of components as number codes, so that the resulting arithmetic has a
direct interaction with the dynamics of the action.

\subsection{A complementary-cover certificate}
We define a predicate $\Buf$. For $A\subseteq V$ in $\DD$, put $\Buf(\varnothing,V)$ true. For nonempty
$A$ define $\Buf(A,V)$ by the existence of $U,X_-,X_+,Y_-,Y_+\in\DD$ satisfying
\begin{align}
 &A\ll U,\label{eq:buf1}\\
 &X_\pm\ll Y_\pm,\qquad Y_\pm\cap A=\varnothing,\label{eq:buf2}\\
 &U\cap V^\perp\subseteq X_-\oplus X_+.\label{eq:buf3}
\end{align}
The last condition means precisely
\begin{equation}\label{eq:buf-syntax}
 \neg\exists Z\in\DD\quad
 \left[Z\ne\varnothing\ \wedge\ Z\subseteq U\ \wedge\
 Z\cap V=Z\cap X_-=Z\cap X_+=\varnothing\right].
\end{equation}
By Boolean density this is equivalent to~\eqref{eq:buf3}. Neither the
complement nor the join is required to be an element of $\DD$. The list of
five auxiliary support variables is fixed.

\begin{lem}\label{lem:buffer-sound}
In every dimension, $\Buf(A,V)$ implies $\cl A\subseteq V$. In particular,
$A$ meets only finitely many components of $V$.
\end{lem}
\begin{proof}
An open set disjoint from $A$ is disjoint from $\cl A$. Hence $\cl X_\pm\subseteq Y_\pm\subseteq
M\setminus\cl A$. Because the join is finite, $\cl(X_-\oplus X_+)=\cl X_-\cup\cl X_+$. Taking
closures in~\eqref{eq:buf3} shows that
$\cl(U\cap V^\perp)$ is disjoint from $\cl A$. On the other hand,
$\cl A\subseteq U$ by~\eqref{eq:buf1}. If $p\in\cl A\setminus V$, regular
openness gives $p\in\cl(V^\perp)$, and the openness of $U$ gives
$p\in\cl(U\cap V^\perp)$, a contradiction. Finally the open components of
$V$ cover the compact set $\cl A$. A finite subcollection covers it, and every
component meeting $A$ belongs to that subcollection.
\end{proof}

In the sequel we write $A\Subset B$ to mean $\overline A\subseteq B$ as a topological fact, distinguishing it from the (related) predicate $\ll$.

\begin{lem}[Two complementary balls]\label{lem:two-balls}
Let $d\geq2$ denote the ambient dimension. Given finitely many prescribed points in a coordinate ball,
one can choose boxes $E_i\Subset D_i$ around these points, with the $D_i$ arbitrarily small, such that the following
holds. If $K_0$ is a compact box in the chart containing all the $D_i$, then
there are open $d$-balls $O_-$ and $O_+$ satisfying
\[
O_\pm\cap E_i=\varnothing
\quad\text{for every }i,
\]
and
\[
K_0\setminus\bigcup_i\Int D_i\subseteq O_-\cup O_+.
\]
\end{lem}

\begin{proof}
Change the coordinates, by a compactly supported homeomorphism of the
coordinate ball, so that the prescribed points lie in
$\R^{d-1}\times\{0\}$ and have distinct first coordinates. This change of
coordinates need not belong to $\GG$. The assumption $d\geq2$ is used
here to provide at least one horizontal coordinate.

Choose pairwise disjoint horizontal boxes $P_i\Subset P_i^+$ and $a>0$ so that $E_i\subset
P_i\times(-a/2,a/2)$ and $\overline{P_i^+}\times[-2a,2a]\subseteq\Int D_i$. By taking the $P_i^+$
and $a$ sufficiently small, the $D_i$ may be chosen
inside arbitrarily small prescribed neighborhoods of the original points.

Choose an open $(d-1)$-box $Q$ and $H>2a$ such that $K_0\subset Q\times(-H,H)$. There is a continuous function
\[
\phi:Q\longrightarrow[-a,H/2]
\]
such that
\[
\phi=-a\quad\text{on a neighborhood of }\bigcup_i\overline{P_i},
\qquad
\{x\in Q:\phi(x)\leq0\}\Subset\bigcup_iP_i^+.
\]
Set
\begin{align*}
O_-&=\{(x,t):x\in Q,\ -H<t<\phi(x)\},\\
O_+&=\{(x,t):x\in Q,\ -\phi(x)<t<H\}.
\end{align*}
Each of $O_-$ and $O_+$ is an open $d$-ball. For instance,
\[
(x,s)\longmapsto
\bigl(x,-H+s(H+\phi(x))\bigr)
\]
is a homeomorphism from $Q\times(0,1)$ onto $O_-$, and the analogous
formula gives a parametrization of $O_+$.

Since $\phi=-a$ near every $\overline{P_i}$, both $O_-$ and $O_+$ are
disjoint from every $E_i$. If $(x,t)\in K_0$ belongs to neither $O_-$ nor
$O_+$, then $\phi(x)\leq t\leq-\phi(x)$. Consequently
\[
\phi(x)\leq0,
\qquad
|t|\leq-\phi(x)\leq a.
\]
Hence $x\in P_i^+$ for some $i$, and therefore $(x,t)\in P_i^+\times[-a,a] \subseteq
\overline{P_i^+}\times[-2a,2a] \subseteq\Int D_i$. Thus
\[
K_0\setminus\bigcup_i\Int D_i\subseteq O_-\cup O_+,
\]
as desired.
\end{proof}

\begin{prop}\label{prop:buffer-exists}
Suppose $\dim M\geq2$, let $V\in\DD$, and let $\mathcal C$ be a finite
subfamily of the components of $V$. Then there is $A\in\DD$ such that
\[
\cc_V(A)=\mathcal C
\qquad\text{and}\qquad
\Buf(A,V).
\]
\end{prop}

\begin{proof}
If $\mathcal C=\varnothing$, take $A=\varnothing$. Otherwise, choose one
interior point in each component belonging to $\mathcal C$. By
Lemma~\ref{lem:finite-ball}, these finitely many points lie in a single
interior chart.

Apply Lemma~\ref{lem:two-balls}, choosing the $D_i$ so that each $D_i$ is
compactly contained in the corresponding component of $V$. For each $i$,
choose a nontrivial local element whose extended support is compactly
contained in $E_i$, with these extended supports having pairwise disjoint
closures. Let $A$ be the extended support of the product of these elements.
Then
$\cc_V(A)=\mathcal C$.

Choose $U\in\DD$ with $A\ll U$ and with $\cl U\subseteq K_0$ for some sufficiently large compact box
$K_0$ in the same chart. Set $K=\cl U\cap(M\setminus V)$. Since every $D_i$ is contained in $V$,
\[
K\cap D_i=\varnothing
\quad\text{for every }i.
\]
Lemma~\ref{lem:two-balls} therefore gives
$K\subseteq O_-\cup O_+$.

Since $K$ is compact, there are compact sets $K_\pm\subseteq O_\pm$ such that $K=K_-\cup K_+$.
Applying Lemma~\ref{lem:large-subset} to each $K_\pm$ gives $K_\pm\subseteq X_\pm\ll Y_\pm\subseteq
O_\pm$. Because each $O_\pm$ is disjoint from every $E_i$, it is disjoint from $A$.
Thus $Y_\pm$ is disjoint from $A$, while $U\cap V^\perp \subseteq K \subseteq X_-\cup X_+$. Hence
$X_\pm,Y_\pm$ have all the properties required in the definition of
$\Buf(A,V)$.

If one of the sets $K_\pm$ is empty and
Lemma~\ref{lem:large-subset} is stated only for nonempty compact sets,
choose instead any $X_\pm,Y_\pm\in\DD$ satisfying $X_\pm\ll Y_\pm\subseteq O_\pm$. This completes the proof.
\end{proof}

Define
\begin{equation}\label{eq:finfoot}
\begin{aligned}
 \FinFoot(W,V)\quad\Longleftrightarrow\quad&
 W\subseteq V\\
 &{}\wedge\ \exists A\in\DD\,
 [\SameFoot(A,W;V)\wedge\Buf(A,V)].
\end{aligned}
\end{equation}

\begin{cor}\label{cor:finfoot}
In every dimension, $\FinFoot(W,V)$ implies that $\cc_V(W)$ is finite. In
dimensions at least two the converse holds, for arbitrary elements $V\in\DD$.
\end{cor}
\begin{proof}
One direction is Lemma~\ref{lem:buffer-sound}, and the converse is
Proposition~\ref{prop:buffer-exists}.
\end{proof}

\begin{lem}[The one-dimensional bound]\label{lem:buffer-one}
Suppose $\dim M=1$. If $A,V\in\DD$ and $\Buf(A,V)$, then $\#\cc_V(A)\leq 3$. Write $\#\cc_V(A)\geq4$
for the first-order condition that there are four pairwise distinct components of $V$ which meet
$A$. Then the sentence
\begin{equation}\label{eq:H} H_2:=\exists A,V\in\DD\quad \bigl[\Loc(V)\wedge\Buf(A,V)\wedge\#\cc_V(A)\geq4\bigr] 
\end{equation} holds exactly in dimensions at least two, among the locally approximating actions under consideration.
\end{lem}
\begin{proof}
Suppose first that $\dim M=1$. The case $A=\varnothing$ is immediate.
Otherwise, write $\cc_V(A)=\{C_1,\ldots,C_k\}$, and let
$U,X_\pm,Y_\pm$ witness $\Buf(A,V)$. Choose $p_i\in A\cap C_i$ and order
these points along the collared interval $I\subseteq U$ containing $A$
supplied by $A\ll U$. For each $i<k$, there is a nonempty open interval
$J_i\subseteq I\cap V^\perp$ between $p_i$ and $p_{i+1}$. Otherwise $V$
would be dense between these points, and regular openness would imply
$C_i=C_{i+1}$.

Each nonempty $X_\pm$ lies in an interval contained in $Y_\pm$ and avoiding all
$p_i$, since $X_\pm\ll Y_\pm$ and $Y_\pm\cap A=\varnothing$. It therefore
meets at most one of the gaps $J_i$. On the other hand,
$J_i\subseteq U\cap V^\perp\subseteq X_-\oplus X_+$, and $X_-\cup X_+$
is dense in its Boolean join. Thus every $J_i$ meets $X_-\cup X_+$.
Consequently $k-1\leq2$, and $k\leq3$. On a circle, the gaps $J_i$ still
lie in distinct components of $M\setminus\{p_1,\ldots,p_k\}$; the
complementary arc does not change this count.

Now suppose $\dim M\geq2$. In a coordinate ball, take the product of four
nontrivial elements supported in balls with pairwise disjoint closures,
and let $V$ be its extended support. Then $\Loc(V)$ holds, and $V$ has at
least four distinct components $C_1,C_2,C_3,C_4$. Apply
Proposition~\ref{prop:buffer-exists} to this four-element family to obtain
$A$ with $\cc_V(A)=\{C_1,C_2,C_3,C_4\}$ and $\Buf(A,V)$. Hence
\eqref{eq:H} holds.
\end{proof}

\subsection{Finite strings of components}
A tuple $(U,V,W,f,W_0,W_t)$ is a \emph{positive string code} if
$U,V,W,W_0,W_t\in\DD$ and 
\begin{equation}\label{eq:string-basic} \Loc(U),\qquad V\ll U,\qquad f\in\GG[U],\qquad \varnothing\neq W_0,W_t\subseteq W\subseteq V. 
\end{equation} 
We further require that there are components $C_0,C_t$ of $V$ such that 
\begin{equation}\label{eq:string-endpoints} \cc_V(W_0)=\{C_0\}, \qquad \cc_V(W_t)=\{C_t\}.
\end{equation} 
Thus $C_0$ and $C_t$ are uniquely determined by $W_0$ and $W_t$. The reader should picture the
components in $\cc_V(W)$ as arranged in a directed string
\[ C_0\longrightarrow C_1\longrightarrow\cdots\longrightarrow C_t, \qquad C\longmapsto f(C). \]
The conditions below say that $f$ carries every selected component except $C_t$ to another selected
component, that $f(C_t)$ is disjoint from $W$, and that starting with $C_0$ and
repeatedly applying this successor operation exhausts all of $\cc_V(W)$.

We first impose the finiteness condition
\begin{equation}\label{eq:string-finite-guard} H_2\ \Longrightarrow\ \FinFoot(W,V). 
\end{equation} 
The remaining conditions are 
\begin{enumerate} 
\item\label{string:successor} 
For every $C\in\cc_V(W)$, 
\[ C\neq C_t \quad\Longrightarrow\quad f(C)\in\cc_V(W). \]
 
\item\label{string:terminal}  $f(C_t)\cap W=\varnothing$. 
\item\label{string:induction} 
For every $A\in\DD$ with $A\subseteq V$, the following condition holds:

\begin{equation}\label{eq:string-induction-formula} \left[ \begin{aligned} &\cc_V(A)\subseteq\cc_V(W),\\ &C_0\in\cc_V(A),\\ &\forall C\in\cc_V(A)\, \bigl(C\neq C_t\Longrightarrow f(C)\in\cc_V(A)\bigr) \end{aligned} \right] \quad\Longrightarrow\quad \cc_V(A)=\cc_V(W). 
\end{equation} 
\end{enumerate} 

Condition~\ref{string:induction} is the induction condition for the successor map $C\mapsto f(C)$.
Namely, among the subcollections of $\cc_V(W)$ represented by sets $A\in\DD$, the only one which
contains $C_0$ and is closed under
\[ C\longmapsto f(C) \qquad (C\neq C_t) \]
 is $\cc_V(W)$ itself. In particular, the condition rules out a proper initial part, a disjoint
cycle, or any other proper subcollection closed under the prescribed successor relation. All of
these conditions are fixed first-order formulae by Proposition~\ref{prop:interp-conn}. Here
expressions such as
\[ \cc_V(A)\subseteq\cc_V(W) \qquad\text{and}\qquad C\in\cc_V(A) \]
 are abbreviations for the corresponding first-order relations on components supplied by that
proposition. Let $\Str^+(U,V,W,f,W_0,W_t)$ denote the conjunction of
\eqref{eq:string-basic}--\eqref{eq:string-induction-formula}. We define $\Num(W,V)$ by
\[
\begin{aligned}
 \Num(W,V)\quad\Longleftrightarrow\quad& W=V=\varnothing\\
 &{}\vee\ \exists U,W_0,W_t\in\DD\ \exists f\in\GG\,
 \Str^+(U,V,W,f,W_0,W_t).
\end{aligned}
\]
 The pair $(W,V)=(\varnothing,\varnothing)$ represents $0$; every positive code is represented by
the components in $\cc_V(W)$ together with the successor map induced by $f$.

\begin{lem}\label{lem:fin-comp}
Let $(U,V,W,f,W_0,W_t)$ be a positive string code for an action in $\goodclass$, and let
$C_0,C_t$ be determined by~\eqref{eq:string-endpoints}. Then there is a
unique standard integer $n\geq1$ such that $\cc_V(W) = \{C_0,f(C_0),\ldots,f^{n-1}(C_0)\}$, where
$f^{n-1}(C_0)=C_t$. Moreover,
\[
 f^i(C_0)\neq f^j(C_0)
 \qquad
 (0\leq i<j\leq n-1).
\]
\end{lem}

\begin{proof}
Consider the forward orbit $\mathcal O = \{f^j(C_0):j\geq0\}$, as long as its terms belong to $\cc_V(W)$. By
Condition~\ref{string:successor},
\[
 C\in\cc_V(W),\quad C\neq C_t
 \quad\Longrightarrow\quad
 f(C)\in\cc_V(W).
\]

We first show that $C_t=f^m(C_0)$ for some standard $m\geq0$.

Suppose otherwise that
\[
 f^j(C_0)\neq C_t
 \qquad\text{for every }j\geq0.
\]
Then Condition~\ref{string:successor} gives
\[
 \mathcal O\subseteq\cc_V(W)
 \qquad\text{and}\qquad
 f(\mathcal O)\subseteq\mathcal O.
\]

Assume first that $\dim M\geq2$. By
\eqref{eq:string-finite-guard} and Corollary~\ref{cor:finfoot}, $\#\cc_V(W)<\infty$, and hence
$\#\mathcal O<\infty$. By Lemma~\ref{lem:finite-selection}, there is $A\in\DD$ such that
\[
 A\subseteq V,
 \qquad
 \cc_V(A)=\mathcal O.
\]
Thus
\[
 \cc_V(A)\subseteq\cc_V(W),
 \qquad
 C_0\in\cc_V(A),
\]
and, since $C_t\notin\mathcal O$,
\[
 C\in\cc_V(A)
 \quad\Longrightarrow\quad
 f(C)\in\cc_V(A).
\]
Condition~\ref{string:induction} therefore gives $\cc_V(A)=\cc_V(W)$. This is impossible, since
\[
 C_t\in\cc_V(W)
 \qquad\text{but}\qquad
 C_t\notin\cc_V(A).
\]

Now assume that $\dim M=1$. The component-selection hypothesis gives
$A\in\DD$ directly with
\[
 A\subseteq V,
 \qquad
 \cc_V(A)=\mathcal O.
\]
The same application of Condition~\ref{string:induction} yields $\cc_V(A)=\cc_V(W)$, again
contradicting $C_t\notin\mathcal O$. Hence in every dimension under consideration there is some
$m\geq0$ such
that
$f^m(C_0)=C_t$.

Let $m$ be least with this property. Then $C_0,f(C_0),\ldots,f^m(C_0)$ are pairwise distinct. Indeed, if
\[
 f^i(C_0)=f^j(C_0)
 \qquad
 (0\leq i<j\leq m),
\]
then the forward orbit from $f^i(C_0)$ is periodic. Since
\[
 f^r(C_0)\neq C_t
 \qquad (r<m),
\]
such an orbit cannot reach $C_t$ for the first time at $m$, a
contradiction.

Set $\mathcal P = \{C_0,f(C_0),\ldots,f^m(C_0)\}$. By Lemma~\ref{lem:finite-selection}, choose $A\in\DD$ such that
\[
 A\subseteq V,
 \qquad
 \cc_V(A)=\mathcal P.
\]
We have
\[
 \mathcal P\subseteq\cc_V(W),
 \qquad
 C_0\in\mathcal P,
\]
and for every $C\in\mathcal P$,
\[
 C\neq C_t
 \quad\Longrightarrow\quad
 f(C)\in\mathcal P.
\]
Condition~\ref{string:induction} now gives $\mathcal P=\cc_V(W)$. Condition~\ref{string:terminal}
also gives $f(C_t)\cap W=\varnothing$, so
$f(C_t)\notin\cc_V(W)$.

Finally, put $n=m+1$. Then
\[
 \cc_V(W)
 =
 \{C_0,f(C_0),\ldots,f^{n-1}(C_0)\},
 \qquad
 f^{n-1}(C_0)=C_t,
\]
and the components displayed above are pairwise distinct. The least
integer $m$ satisfying $f^m(C_0)=C_t$ is unique, and hence so is $n$.
\end{proof}

\begin{lem}\label{lem:ext-exist}
For every standard $n\in\N$, a number code of size $n$ exists inside any
prescribed nonempty interior coordinate ball.
\end{lem}

\begin{proof}
For $n=0$, take
$(W,V)=(\varnothing,\varnothing)$.

Fix $n\geq1$. Inside the prescribed coordinate ball, choose
$U\in\DD$ with $\Loc(U)$ and a collared closed $d$-ball $B\Subset U$. Choose pairwise disjoint
closed balls $B_0,\ldots,B_n\subseteq\Int B$ arranged in this order along an embedded arc. There is
a homeomorphism
$h$ supported in $\Int B$ such that
\[
 h(B_i)\subseteq\Int B_{i+1}
 \qquad
 (0\leq i<n).
\]
By local approximation, we may choose $f\in\GG[U]$ such that
\[
 f(B_i)\subseteq\Int B_{i+1}
 \qquad
 (0\leq i<n).
\]

Choose a nonempty $Y\in\DD$ with $\overline Y\subseteq\Int B_0$. Then, for $0\leq i\leq n$,
$\overline{f^i(Y)}\subseteq\Int B_i$, and therefore
\[
 \overline{f^i(Y)}
 \cap
 \overline{f^j(Y)}
 =
 \varnothing
 \qquad
 (i\neq j).
\]
Set
\[
 V=\bigoplus_{i=0}^{n}f^i(Y).
\]
By taking the corresponding finite product of conjugates, we have
\[
 V\in\DD,
 \qquad
 V\ll U.
\]

Choose a component $C\in\cc(Y)$. Then $C,f(C),\ldots,f^n(C)$ lie in pairwise distinct components of $V$. By
Lemma~\ref{lem:finite-selection}, choose $W\in\DD$ such that $W\subseteq V$ and $\cc_V(W) =
\{C,f(C),\ldots,f^{n-1}(C)\}$. Choose nonempty $W_0,W_t\in\DD$ satisfying
\[
 W_0\subseteq W\cap C,
 \qquad
 W_t\subseteq W\cap f^{n-1}(C).
\]
Thus
\[
 C_0=C,
 \qquad
 C_t=f^{n-1}(C).
\]

For $0\leq i<n-1$, $f\bigl(f^i(C)\bigr)=f^{i+1}(C)\in\cc_V(W)$, so Condition~\ref{string:successor}
holds. Moreover, $f(C_t)=f^n(C)$, and $f^n(C)\notin\cc_V(W)$, whence $f(C_t)\cap W=\varnothing$.
Thus Condition~\ref{string:terminal} holds.

It remains to verify Condition~\ref{string:induction}. Let $A\in\DD$ satisfy
\[
 A\subseteq V,
 \qquad
 \cc_V(A)\subseteq\cc_V(W),
 \qquad
 C_0\in\cc_V(A),
\]
and suppose that
\[
 C'\in\cc_V(A),\quad C'\neq C_t
 \quad\Longrightarrow\quad
 f(C')\in\cc_V(A).
\]
Since $C=C_0\in\cc_V(A)$, successive applications of the last implication give
\[
 f^i(C)\in\cc_V(A)
 \qquad
 (0\leq i\leq n-1).
\]
Hence $\cc_V(W) \subseteq \cc_V(A)$. Together with $\cc_V(A)\subseteq\cc_V(W)$, this yields
$\cc_V(A)=\cc_V(W)$. Therefore Condition~\ref{string:induction} holds.

Finally, if $H_2$ holds, Corollary~\ref{cor:finfoot} and
$\#\cc_V(W)=n<\infty$ give $\FinFoot(W,V)$. Thus~\eqref{eq:string-finite-guard} also holds, and
$(U,V,W,f,W_0,W_t)$ is a positive string code of size $n$.
\end{proof}

For each fixed standard $n$, the expressions $f^0,\ldots,f^n$ and the finite products used above
occur only in the external construction
of witnesses. They are not part of the formula defining $\Num(W,V)$.
Thus the formula defining the domain has fixed arity, while
Lemma~\ref{lem:ext-exist} establishes externally that codes exist for every
standard
$n\in\N$.

\subsection{The equivalence relation and arithmetic operations}

Let $A,V,V'\in\DD$, with $A\subseteq V$, and let
\[
 C\in\cc(V),
 \qquad
 D\in\cc(V').
\]
Define $ \Carry(h,A;C,D)$ to mean
$ h(A\cap C)\subseteq D$.
Since $A\cap C$ need not belong to $\DD$, this is expressed by
\begin{equation}\label{eq:carry}
 \forall B\in\DD\,
 \bigl[
 B\subseteq A\ \wedge\ B\subseteq C
 \ \Longrightarrow\
 h(B)\subseteq D
 \bigr].
\end{equation}
By Proposition~\ref{prop:interp-conn}, the relations
$B\subseteq C$ and $h(B)\subseteq D$ are first-order relations between
support variables and component variables. Thus $\Carry$ is a fixed
first-order formula.

For $W,V,W',V'\in\DD$ with $W\subseteq V$ and $W'\subseteq V'$, define $ \Inj(W,V;W',V')$ to hold if
there are $Z\in\DD$ and $h\in\GG$ such that
\begin{equation}\label{eq:inj-source}
 Z\subseteq V,
 \qquad
 \SameFoot(Z,W;V),
\end{equation}
and the following conditions hold:
\begin{align}
 &\forall C\in\cc_V(Z)\ 
   \exists D\in\cc_{V'}(W')\
   \Carry(h,Z;C,D),                                     
   \label{eq:inj-existence}\\
 &\forall C_1,C_2\in\cc_V(Z)\
   \forall D_1,D_2\in\cc_{V'}(W')                       \notag\\
 &\qquad
 \bigl[
   \Carry(h,Z;C_1,D_1)
   \wedge
   \Carry(h,Z;C_2,D_2)
   \wedge
   C_1\neq C_2
 \bigr]
 \Longrightarrow
 D_1\neq D_2.
 \label{eq:inj-injective}
\end{align}
If $C\in\cc_V(Z)$, then $Z\cap C\neq\varnothing$. Hence
$h(Z\cap C)\neq\varnothing$, and two distinct components of $V'$ cannot
both contain $h(Z\cap C)$. Consequently the component $D$ in
\eqref{eq:inj-existence} is unique.

Define $ \Bij(W,V;W',V')$ by
$ \Inj(W,V;W',V')$,
together with
\begin{equation}\label{eq:bij-surjective}
 \forall D\in\cc_{V'}(W')\
 \exists C\in\cc_V(Z)\
 \Carry(h,Z;C,D),
\end{equation}
for the same witnesses $Z$ and $h$.

Suppose that
\[
 \#\cc_V(W)<\infty,
 \qquad
 \#\cc_{V'}(W')<\infty.
\]
Then
\begin{align}
 \Inj(W,V;W',V')
 &\quad\Longleftrightarrow\quad
 \#\cc_V(W)\leq\#\cc_{V'}(W'),                         \label{eq:inj-card}\\
 \Bij(W,V;W',V')
 &\quad\Longleftrightarrow\quad
 \#\cc_V(W)=\#\cc_{V'}(W').                            \label{eq:bij-card}
\end{align}
The forward implications follow from the maps on finite component sets
defined by~\eqref{eq:inj-existence}. Conversely, choose an injection or
bijection between these component families. In dimension one, first
choose a target subfamily of the required size and then an increasing
correspondence on an interval, or a cyclic-order-preserving correspondence
on a circle. Lemma~\ref{lem:finite-transport} supplies both $Z\in\DD$ with
$Z\subseteq V$ and $\cc_V(Z)=\cc_V(W)$, and $h\in\GG$ realizing this
correspondence on the selected pieces. For an empty source, take
$Z=\varnothing$ and $h=1$; in the bijective case the target is then empty
as well.

For number codes, define
\begin{equation}\label{eq:number-equivalence}
 (W,V)\sim(W',V')
 \quad\Longleftrightarrow\quad
 \Bij(W,V;W',V').
\end{equation}
By Lemmas~\ref{lem:fin-comp} and~\ref{lem:ext-exist}, every positive number
code has a finite nonempty set $\cc_V(W)$, every finite positive cardinality
occurs, and
\[
 (W,V)\sim(W',V')
 \quad\Longleftrightarrow\quad
 \#\cc_V(W)=\#\cc_{V'}(W').
\]
Together with the code $ (\varnothing,\varnothing)$ for $0$, the quotient
by~\eqref{eq:number-equivalence} is therefore naturally
identified with
$\N$.

We next extend finite counting to arbitrary definable families of components.
Let $\nu$ be a number in the interpreted copy of $\N$, and let
$A,V\in\DD$ with $A\subseteq V$. Define $\Card(\nu;A,V)$ by choosing a representative
$(W_\nu,V_\nu)$ of $\nu$ and requiring
\begin{equation}\label{eq:Card-def}
 \Bij(W_\nu,V_\nu;A,V).
\end{equation}
This is independent of the representative of $\nu$. In particular,
\begin{equation}\label{eq:Card}
 \Card([n];A,V)
 \quad\Longleftrightarrow\quad
 \#\cc_V(A)=n.
\end{equation}
No hypothesis that $(A,V)$ is itself a number code is imposed. Indeed,
surjectivity in~\eqref{eq:bij-surjective}, together with finiteness of the
source, implies $\#\cc_V(A)<\infty$. The same construction applies when $\cc_V(A)$ is replaced by any
first-order specified subfamily of $\cc(V)$.

The element $0$ is represented by
$ (\varnothing,\varnothing)$,
and $1$ is represented by any number code $(W,V)$ satisfying
$ \#\cc_V(W)=1$.
For $\nu,\mu\in\N$, define $ \nu\leq\mu$ by
$ \Inj(W_\nu,V_\nu;W_\mu,V_\mu)$,
for representatives of $\nu$ and $\mu$. By~\eqref{eq:inj-card}, this gives
the usual ordering of $\N$.

For addition, let $ (W_\ell,V_\ell)$ represent $\ell$. We define $ \Add([n],[m],[\ell])$ by the
existence of $B,C\in\DD$, with $B,C\subseteq V_\ell$, such that
\begin{equation}\label{eq:addition-partition}
 \cc_{V_\ell}(W_\ell)
 =
 \cc_{V_\ell}(B)\sqcup\cc_{V_\ell}(C),
\end{equation}
and
\begin{equation}\label{eq:addition}
 \Card([n];B,V_\ell),
 \qquad
 \Card([m];C,V_\ell).
\end{equation}
The equality~\eqref{eq:addition-partition} abbreviates the componentwise
conditions
\begin{align*}
 &\cc_{V_\ell}(B)\subseteq\cc_{V_\ell}(W_\ell),\\
 &\cc_{V_\ell}(C)\subseteq\cc_{V_\ell}(W_\ell),\\
 &\cc_{V_\ell}(B)\cap\cc_{V_\ell}(C)=\varnothing,\\
 &\forall D\in\cc_{V_\ell}(W_\ell)\,
 \bigl[
 D\in\cc_{V_\ell}(B)\ \vee\
 D\in\cc_{V_\ell}(C)
 \bigr].
\end{align*}
Hence
\[
 \Add([n],[m],[\ell])
 \quad\Longleftrightarrow\quad
 \ell=n+m.
\]
For the converse implication, finite selection gives $B$ and $C$ realizing
any prescribed partition of the finite set $\cc_{V_\ell}(W_\ell)$ into
subsets of cardinalities $n$ and $m$. Notice that the construction uses a
partition of a single finite set of components. It does not require
disjoint regular open sets to have disjoint closures, nor does it require
their join to preserve their components.

We now define multiplication. Let $(W_n,V_n)$ represent $n$. Choose
\[
 P,Q\in\DD,
 \qquad
 P\subseteq Q\subseteq V_n,
\]
and require
\begin{equation}\label{eq:fibers-over-selected}
 D\in\cc_Q(P)
 \quad\Longrightarrow\quad
 \exists C\in\cc_{V_n}(W_n)\quad D\subseteq C.
\end{equation}
For $C\in\cc_{V_n}(W_n)$, set $\mathcal F_C(P,Q) = \{D\in\cc(Q):D\subseteq C,\ D\cap
P\neq\varnothing\}$. Define $\FiberCard([m];P,Q,C)$ by $\#\mathcal F_C(P,Q)=m$, using
\eqref{eq:Card-def} with this first-order specified family in place of
$\cc_V(A)$.
Then define $ \Mult([n],[m],[\ell])$ by
\begin{equation}\label{eq:multiplication}
 \forall C\in\cc_{V_n}(W_n)\
 \FiberCard([m];P,Q,C),
 \qquad
 \Card([\ell];P,Q).
\end{equation}

The sets
\[
 \mathcal F_C(P,Q),
 \qquad
 C\in\cc_{V_n}(W_n),
\]
are pairwise disjoint, and~\eqref{eq:fibers-over-selected} gives
\[
 \cc_Q(P)
 =
 \bigsqcup_{C\in\cc_{V_n}(W_n)}
 \mathcal F_C(P,Q).
\]
Thus~\eqref{eq:multiplication} implies
\[
 \ell
 =
 \#\cc_Q(P)
 =
 \sum_{C\in\cc_{V_n}(W_n)}
 \#\mathcal F_C(P,Q)
 =
 nm.
\]

Conversely, suppose $n,m>0$, and write
$\cc_{V_n}(W_n)=\{C_1,\ldots,C_n\}$. Choose interior balls
$B_i$ with pairwise disjoint compact closures in $C_i$. By
Lemma~\ref{lem:ext-exist}, choose a size-$m$ number code $(P_i,Q_i)$ in
$B_i$, with $\cl Q_i\subseteq B_i$. Set
\[
 P=\bigoplus_{i=1}^n P_i,\qquad Q=\bigoplus_{i=1}^n Q_i.
\]
These are extended supports of finite products, since the closures of the
pieces are pairwise disjoint. No components merge between the different
$Q_i$. Exactly $m$ components of $Q_i$ meet $P_i$, so
$\#\mathcal F_{C_i}(P,Q)=m$ for every $i$ and $\#\cc_Q(P)=nm$.
The sets $Q_i$, and hence $Q$, may have additional components not meeting
$P$; the formulae do not count them. Thus~\eqref{eq:multiplication} holds
with $\ell=nm$. If $n=0$ or $m=0$, take $P=Q=\varnothing$.

\begin{cor}\label{cor:arith-gen}
For every action in $\goodclass$, the quotient of the number-code domain by
$\sim$, equipped with the relations defined above, is a uniform
parameter-free interpretation of
\[
 (\N,+,\cdot,\leq,0,1).
\]
\end{cor}

\begin{proof}
By Lemmas~\ref{lem:fin-comp} and~\ref{lem:ext-exist} and
\eqref{eq:number-equivalence}, the quotient of the definable domain by
$\sim$ is naturally identified with $\N$. Equations
\eqref{eq:inj-card}, \eqref{eq:addition}, and
\eqref{eq:multiplication} show respectively that the interpreted relations
are
\[
 \leq,\qquad +,\qquad\cdot.
\]
The definitions depend only on the equivalence classes of the chosen
representatives. Thus the interpretation is uniform and parameter-free.
This proves Theorem~\ref{thm:arith}.
\end{proof}

\subsection{Relating arithmetic to the action structure}
\label{subsec:exp}

We now use the interpreted arithmetic to count finite families of components
which are not themselves required to be number codes. In particular, the
sets $A,W\in\DD$ occurring below are arbitrary; finiteness of
$\cc_W(A)$ is imposed by a formula of the form
$\Card(k;A,W)$.

Let $k>0$. Define $\Path(f,k;A,W,C_s,C_t)$ by the conjunction of
\begin{align}
 &\Card(k;A,W),                                           \label{eq:path-card}\\
 &C_s,C_t\in\cc_W(A),                                    \label{eq:path-endpoints}\\
 &\forall C\in\cc_W(A)\,
 \bigl[
 C\neq C_t\Longrightarrow f(C)\in\cc_W(A)
 \bigr],                                                 \label{eq:path-successor}\\
 &f(C_t)\cap A=\varnothing,                              \label{eq:path-terminal}
\end{align}
together with
\begin{equation}\label{eq:path-induction}
 \begin{split}
 \forall B\in\DD\,
 \Bigl[
 &B\subseteq W
 \ \wedge\
 \cc_W(B)\subseteq\cc_W(A)
 \ \wedge\
 C_s\in\cc_W(B)\\
 &{}\wedge\
 \forall C\in\cc_W(B)\,
 \bigl(
 C\neq C_t\Longrightarrow f(C)\in\cc_W(B)
 \bigr)
 \Bigr]\\
 &\hspace{35mm}\Longrightarrow\
 \cc_W(B)=\cc_W(A).
 \end{split}
\end{equation}
Since~\eqref{eq:path-card} makes $\cc_W(A)$ finite, the argument of
Lemma~\ref{lem:fin-comp} gives $\cc_W(A) = \{C_s,f(C_s),\ldots,f^{k-1}(C_s)\}$, with
$ f^{k-1}(C_s)=C_t$,
and these $k$ components are pairwise distinct.

Similarly, define $ \Cycle(f,k;A,W,C_s)$ by
\begin{align}
 &\Card(k;A,W),                                          \label{eq:cycle-card}\\
 &C_s\in\cc_W(A),                                       \label{eq:cycle-start}\\
 &\forall C\in\cc_W(A)\quad
 f(C)\in\cc_W(A),                                       \label{eq:cycle-successor}
\end{align}
and
\begin{equation}\label{eq:cycle-induction}
 \begin{split}
 \forall B\in\DD\,
 \Bigl[
 &B\subseteq W
 \ \wedge\
 \cc_W(B)\subseteq\cc_W(A)
 \ \wedge\
 C_s\in\cc_W(B)\\
 &{}\wedge\
 \forall C\in\cc_W(B)\quad
 f(C)\in\cc_W(B)
 \Bigr]
 \Longrightarrow
 \cc_W(B)=\cc_W(A).
 \end{split}
\end{equation}
The forward orbit of $C_s$ is a finite nonempty subset of $\cc_W(A)$ which
is closed under $f$. Finite selection and~\eqref{eq:cycle-induction} imply
that this orbit is all of $\cc_W(A)$. Consequently $\cc_W(A) = \{C_s,f(C_s),\ldots,f^{k-1}(C_s)\}$,
the displayed components are pairwise distinct, and
$ f^k(C_s)=C_s$.

Suppose
\[
 \Cycle(f,k;A,W,C_s),
 \qquad
 0\leq r<k.
\]
Define $\At(f,r;A,W,C_s,C_r)$ as follows. If
$ r+1<k$,
require the existence of $B\in\DD$ such that
\begin{equation}\label{eq:at-path}
 B\subseteq W,
 \qquad
 \cc_W(B)\subseteq\cc_W(A),
 \qquad
 \Path(f,r+1;B,W,C_s,C_r).
\end{equation}
If
$ r+1=k$,
require
\begin{equation}\label{eq:at-last}
 C_r\in\cc_W(A),
 \qquad
 f(C_r)=C_s.
\end{equation}
Then
\begin{equation}\label{eq:at-correct}
 \At(f,r;A,W,C_s,C_r)
 \quad\Longleftrightarrow\quad
 C_r=f^r(C_s).
\end{equation}
The second clause also covers
\[
 k=1,\qquad r=0.
\]

We next define $\Forward(f,n;U,V)$ for $n>0$. The definition has the form
\begin{equation}\label{eq:forward-main}
\begin{aligned}
 \forall U_0\in\DD\quad\Bigl[\varnothing\neq U_0\ll U
 &\Longrightarrow\ \exists U_1\in\DD\,\Bigl(\varnothing\neq U_1\ll U_0\\
 &\qquad{}\wedge\bigl[\mathsf P_n(U_1,V)\vee\mathsf C_n(U_1,V)\bigr]
 \Bigr)\Bigr].
\end{aligned}
\end{equation}

The first alternative $\mathsf P_n(U_1,V)$ is
\begin{equation}\label{eq:forward-path}
\begin{aligned}
 \exists W\in\DD\quad\Bigl[U_1\subseteq W\ \wedge\ {}&
 \forall E\in\cc(U_1)\ \exists A\in\DD\\
 &\exists C_s,C_t\in\cc(W)\,
 \bigl[\Path(f,n+1;A,W,C_s,C_t)\\
 &\hspace{12mm}{}\wedge E\subseteq C_s\wedge C_t\subseteq V\bigr]
 \Bigr].
\end{aligned}
\end{equation}

The second alternative $\mathsf C_n(U_1,V)$ is the existence of $ q,m,r\in\N$ such that
\begin{equation}\label{eq:division}
 1\leq m\leq n+1,
 \qquad
 n=qm+r,
 \qquad
 0\leq r<m,
\end{equation}
together with $W\in\DD$, $U_1\subseteq W$, such that
\begin{equation}\label{eq:forward-cycle}
\begin{aligned}
 \forall E\in\cc(U_1)\ \exists A\in\DD\quad&
 \exists C_s,C_r\in\cc(W)\,
 \bigl[\Cycle(f,m;A,W,C_s)\\
 &{}\wedge\At(f,r;A,W,C_s,C_r)\\
 &{}\wedge E\subseteq C_s\wedge C_r\subseteq V\bigr].
\end{aligned}
\end{equation}

Finally, define
\begin{equation}\label{eq:exp-zero}
 \Exp(f,0;U,V)\quad\Longleftrightarrow\quad U=V.
\end{equation}
For $n>0$, define
\begin{equation}\label{eq:exponentiation}
\begin{aligned}
 \Exp(f,n;U,V)
 \quad\Longleftrightarrow\quad&
 \Forward(f,n;U,V)\\
 &{}\wedge\ \Forward(f^{-1},n;V,U).
\end{aligned}
\end{equation}

\begin{lem}\label{lem:exponentiation}
For all $f\in\GG$, $n\in\N$, and $U,V\in\DD$,
\[
 \Exp(f,n;U,V)
 \quad\Longleftrightarrow\quad
 f^n(U)=V.
\]
\end{lem}

\begin{proof}
For $n=0$, this is~\eqref{eq:exp-zero}. Assume henceforth that $n>0$.

Suppose first that
$ f^n(U)=V$.
Let
$ \varnothing\neq U_0\ll U$.
We verify~\eqref{eq:forward-main}.

Assume first that there is $x\in U_0$ such that $x,f(x),\ldots,f^{n+1}(x)$ are pairwise distinct.
Choose an open neighborhood $O$ of $x$ such that $\overline O, f(\overline O), \ldots,
f^{n+1}(\overline O)$ are pairwise disjoint. Take $O$ to be a sufficiently small interior
collared open ball with $\cl O\subseteq U_0$, and choose
\[
 \varnothing\neq U_1\in\DD,
 \qquad
 \overline{U_1}\subseteq O.
\]
Since $U_1\subseteq U$ and $f^n(U)=V$,
\begin{equation}\label{eq:path-image-in-v}
 f^n(U_1)\subseteq V.
\end{equation}
Set
\[
 W=\bigoplus_{j=0}^{n}f^j(U_1).
\]
The sets $f^j(U_1)$ have pairwise disjoint closures, so $W\in\DD$.

Let
$ E\in\cc(U_1)$.
For $0\leq j\leq n$, let $C_j$ be the component of $W$ containing
$f^j(E)$. Then
\[
 C_j=f^j(C_0),
 \qquad
 0\leq j\leq n,
\]
and the $C_j$ are pairwise distinct. Moreover, $C_n\subseteq f^n(U_1)\subseteq V$, while
$f(C_n)\subseteq f^{n+1}(U_1)$ is disjoint from $W$.

By finite selection, choose $A\in\DD$ with $\cc_W(A)=\{C_0,\ldots,C_n\}$. Then
$\Path(f,n+1;A,W,C_0,C_n)$. Since this construction is available for every $E\in\cc(U_1)$,
\eqref{eq:forward-path} holds.

Suppose instead that no point of $U_0$ has $n+2$ pairwise distinct
successive orbit points. Then for every $x\in U_0$ there is $j\in\{1,\ldots,n+1\}$ such that
$ f^j(x)=x$.
Hence
\begin{equation}\label{eq:fix-cover}
 U_0
 \subseteq
 \bigcup_{j=1}^{n+1}\Fix(f^j).
\end{equation}
Each $\Fix(f^j)$ is closed. By the Baire category theorem, there are $ 1\leq m\leq n+1$ and a
nonempty open set $O\subseteq U_0$ such that
$ f^m\mid_O=\id_O$.
Choose $m$ least with this property. For every $1\leq j<m$, $\Fix(f^j)\cap O$ has empty interior in
$O$. After replacing $O$ by a smaller nonempty open
subset, we may therefore assume that every point of $O$ has exact period
$m$.

Write
\[
 n=qm+r,
 \qquad
 0\leq r<m.
\]
Choose $x\in O$. Since $x,f(x),\ldots,f^{m-1}(x)$ are pairwise distinct, we may shrink $O$ further
around $x$ so that $\overline O, f(\overline O), \ldots, f^{m-1}(\overline O)$ are pairwise disjoint and $f^m|_O=\id_O$. We may also take this smaller $O$
to be an interior collared open ball with $\cl O\subseteq U_0$. Choose
\[
 \varnothing\neq U_1\in\DD,
 \qquad
 \overline{U_1}\subseteq O.
\]
On $U_1$ we have $f^n = f^{qm+r} = f^r$. Consequently
\begin{equation}\label{eq:cycle-image-v}
 f^r(U_1)=f^n(U_1)\subseteq V.
\end{equation}

Set
\[
 W=\bigoplus_{j=0}^{m-1}f^j(U_1).
\]
For $E\in\cc(U_1)$, let $C_j$ denote the component of $W$ containing
$f^j(E)$. Then
\[
 C_j=f^j(C_0),
 \qquad
 0\leq j<m,
\]
the $C_j$ are pairwise distinct, and
$ f(C_{m-1})=C_0$.
By finite selection, choose $A\in\DD$ with
$ \cc_W(A)=\{C_0,\ldots,C_{m-1}\}$.
Then $ \Cycle(f,m;A,W,C_0)$,
and by~\eqref{eq:at-correct},
$ \At(f,r;A,W,C_0,C_r)$.
Equation~\eqref{eq:cycle-image-v} gives
$ C_r\subseteq V$.
Thus~\eqref{eq:forward-cycle} holds. We have proved
$ \Forward(f,n;U,V)$.

Applying the same argument to
\[
 f^{-1},
 \qquad
 V,
 \qquad
 U,
\]
and using
$ (f^{-1})^n(V)=U$,
gives
$ \Forward(f^{-1},n;V,U)$.
Therefore,
$ \Exp(f,n;U,V)$.

Conversely, suppose
$ \Forward(f,n;U,V)$.
We first show
\begin{equation}\label{eq:forward-inclusion}
 f^n(U)\subseteq V.
\end{equation}
Indeed, either alternative in the definition of $\Forward$ implies
$ f^n(U_1)\subseteq V$.
For~\eqref{eq:forward-path}, if $E\in\cc(U_1)$ and
\[
 \Path(f,n+1;A,W,C_s,C_t),
 \qquad
 E\subseteq C_s,
 \qquad
 C_t\subseteq V,
\]
then
$ f^n(C_s)=C_t$,
and hence
$ f^n(E)\subseteq V$.
For~\eqref{eq:forward-cycle}, if $ n=qm+r$ and
\[
 \Cycle(f,m;A,W,C_s),
 \qquad
 \At(f,r;A,W,C_s,C_r),
\]
then
\[
 f^m(C_s)=C_s,
 \qquad
 C_r=f^r(C_s),
\]
so $f^n(C_s)=f^{qm+r}(C_s)=f^r(C_s)=C_r\subseteq V$. Again,
$ f^n(E)\subseteq V$.
Since this holds for every component $E$ of $U_1$, we have
$ f^n(U_1)\subseteq V$.

If~\eqref{eq:forward-inclusion} failed, then
$ U\not\subseteq f^{-n}(V)$.
The regular-open Boolean difference $ U\wedge\bigl(f^{-n}(V)\bigr)^\perp$ would be nonzero. Choose $
\varnothing\neq U_0\in\DD$ with
$ U_0\ll
 U\wedge\bigl(f^{-n}(V)\bigr)^\perp$.
By $\Forward(f,n;U,V)$ there is $ \varnothing\neq U_1\ll U_0$ such that
$ f^n(U_1)\subseteq V$.
On the other hand,
$ U_1\subseteq
 \bigl(f^{-n}(V)\bigr)^\perp$,
which implies
$ f^n(U_1)\cap V=\varnothing$,
a contradiction. Hence
$ f^n(U)\subseteq V$.

Now suppose
$ \Exp(f,n;U,V)$.
The second conjunct in~\eqref{eq:exponentiation} gives
$ (f^{-1})^n(V)\subseteq U$.
Applying $f^n$ yields
$ V\subseteq f^n(U)$.
Together with~\eqref{eq:forward-inclusion},
$ f^n(U)=V$.
The same argument includes the cases $U=\varnothing$ or $V=\varnothing$.
\end{proof}

For negative integer exponents one uses $f^{-1}$. For $n\in\N$, define
\begin{equation}\label{eq:power-relation}
 \Pow(f,n,h)
 \quad\Longleftrightarrow\quad
 \forall U\in\DD\
 \Exp(f,n;U,h(U)).
\end{equation}
By Lemma~\ref{lem:exponentiation},
\[
 \Pow(f,n,h)
 \quad\Longleftrightarrow\quad
 \forall U\in\DD\quad
 f^n(U)=h(U).
\]
Thus, with
$ g=h^{-1}f^n$,
we have
\[
 g(U)=U
 \qquad
 \text{for every }U\in\DD.
\]
We claim that this implies
$ g=\id$.
If not, choose $x$ with
$ g(x)\neq x$.
There is an open neighborhood $O$ of $x$ such that
$ O\cap g(O)=\varnothing$.
Choose a nonempty $U\in\DD$ with
$ U\subseteq O$.
Then $ g(U)=U$ by the preceding paragraph, whereas $ g(U)\subseteq g(O)$ and
$ U\subseteq O$,
contradicting
$ O\cap g(O)=\varnothing$.
Therefore
$ h=f^n$.
Hence
\[
 \Pow(f,n,h)
 \quad\Longleftrightarrow\quad
 h=f^n.
\]
This completes the exponentiation part of
Theorem~\ref{thm:member}.

\section{The membership predicate and its consequences}\label{sec:member}

We retain the previous assumption that all actions belong
to $\goodclass$. We have a standard arithmetic sort, a support sort with its
component language, and the exponentiation relation of
Lemma~\ref{lem:exponentiation}.

\subsection{The membership predicate}\label{subsec:word-evaluation}
Fix $r\geq1$. In arithmetic, finite words in abstract generators
$x_1^{\pm1},\ldots,x_r^{\pm1}$ are coded by natural numbers. The predicates
for being a word, its length, and its letter at a specified index are
arithmetical. We take a nonempty word in application order to be
\[
 w=s_{m-1}\cdots s_1s_0,\qquad m\geq1.
\]
For a tuple $\underline g=(g_1,\ldots,g_r)$ the assertion that a group variable
is the letter $s_i(\underline g)$ is a fixed finite disjunction.

Let $w=s_{m-1}\cdots s_0$ be a nonempty word of length $m$. Choose a positive
string code $(U,V,W,f,W_0,W_t)$ of size $m$. Thus, writing
$\cc_V(W)=\{C_0,\ldots,C_{m-1}\}$, we have $C_i=f^i(C_0)$ for
$0\leq i<m$. Choose a nonempty $B\in\DD$ with $B\ll W_0$, and put
$B_i=f^i(B)$ for $0\leq i<m$. Then $\overline{B_i}\subseteq C_i$.
In the defining formula, $B_i$ is specified by $\Exp(f,i;B,B_i)$, and
$C_i$ is the unique component of $V$ containing $B_i$; thus neither
$(B_0,\ldots,B_{m-1})$ nor $(C_0,\ldots,C_{m-1})$ occurs as an
$m$-tuple.

For $Z,S\in\DD$ and $C\in\cc(V)$, define

\[
\Slice(Z,C;S)
\Longleftrightarrow
S\subseteq Z\ \wedge\ S\subseteq C\ \wedge\
\forall X\in\DD\,
\bigl[
X\subseteq Z\wedge X\subseteq C
\Longrightarrow X\subseteq S
\bigr].
\]

Whenever this holds, $S=Z\cap C$. The formula requires $S\in\DD$; there is
no assertion that $Z\cap C$ belongs to $\DD$ for arbitrary $Z$ and $C$.

Fix $a\in\GG$. For $Z\in\DD$, define
\begin{equation}\label{eq:decode}
\begin{aligned}
 \Dec_Z(i,Y)\quad\Longleftrightarrow\quad
 \exists S\in\DD\,\bigl[{}&S\neq\varnothing\ \wedge\ S\subseteq B_i\\
 &{}\wedge\Slice(Z,C_i;S)\ \wedge\Exp(f,i;a(Y),S)\bigr].
\end{aligned}
\end{equation}
If $\Dec_Z(i,Y)$ holds, then $S$ is uniquely determined by $Z$ and $C_i$,
and Lemma~\ref{lem:exponentiation} gives $Y=a^{-1}f^{-i}(S)$. Thus $Y$,
when it exists, is unique. Formally, the notation $B_i,C_i$
in~\eqref{eq:decode} is eliminated by quantifying one support variable and
one component variable and using $\Exp$ together with the component relations
already defined.

Let $Z_s,Z_t\in\DD$ satisfy $Z_s,Z_t\subseteq V$ and
$\cc_V(Z_s),\cc_V(Z_t)\subseteq\cc_V(W)$. For the fixed data above, write
$\Ad_w(Z_s,Z_t)$ for the condition that, for every $i<m$, there are
$Y,T\in\DD$ such that $\Dec_{Z_s}(i,Y)$, $\Dec_{Z_t}(i,T)$, and
$T=s_i(\underline g)(Y)$, and whenever $i+1<m$, also
$\Dec_{Z_s}(i+1,T)$. Equivalently, if $Y_i$ and $T_i$ denote the unique
values determined by the two decoding relations, then
\begin{equation}\label{eq:adapted}
\begin{aligned}
 T_i&=s_i(\underline g)(Y_i) &&(0\leq i<m),\\
 T_i&=Y_{i+1} &&(0\leq i<m-1).
\end{aligned}
\end{equation}
The notation $Y_i,T_i$ in~\eqref{eq:adapted} is only expository; the formula
uses the arithmetic quantifier $i<m$ and the fixed collection of variables
displayed above.

For a nonempty word $w$ of length $m$, define
$\Eval_r(w,\underline g,h)$ by requiring the following for every positive
string code of size $m$, every nonempty $B\ll W_0$, every $a\in\GG$, and
every $Z_s,Z_t$ satisfying $\Ad_w(Z_s,Z_t)$: if
$\Dec_{Z_s}(0,Y)$ and $\Dec_{Z_t}(m-1,T)$, then
\begin{equation}\label{eq:eval}
h(Y)=T.
\end{equation}
For the empty word, set $\Eval_r(w,\underline g,h)\Longleftrightarrow h=1$.

\begin{lem}\label{lem:word-define}
For every group word $w$ in $r$ variables,

\[
\Eval_r(w,\underline g,h)
\quad\Longleftrightarrow\quad
h=w(\underline g).
\]

The formula is parameter-free and uniform in the group and manifold for
each fixed $r$.
\end{lem}

\begin{proof}
The empty word is covered by the definition. Assume henceforth that
$w$ has length $m\geq1$. Suppose first that $h=w(\underline g)$.
For any choices occurring in the
definition of $\Eval_r$, let $Y_i,T_i$ denote the unique values given by
$\Dec_{Z_s}(i,Y_i)$ and $\Dec_{Z_t}(i,T_i)$. Then~\eqref{eq:adapted} gives $T_i=s_i(\underline g)(Y_i)$ for
$0\leq i<m$ and $Y_{i+1}=T_i$ for $0\leq i<m-1$.
Hence, by finite induction,
\[
 T_{m-1}=s_{m-1}(\underline g)\cdots s_0(\underline g)(Y_0)
 =w(\underline g)(Y_0)=h(Y_0).
\]

Thus~\eqref{eq:eval} holds.

Conversely, put $\widehat h=w(\underline g)$ and suppose
$h\neq\widehat h$. Choose a positive string code of size $m$ and a nonempty
$B\ll W_0$, as above. Choose an interior point $p$ such that
$h(p)\neq\widehat h(p)$. Set $h_0=1$ and, for $0\leq i<m$, let
$h_{i+1}=s_i(\underline g)h_i$. Thus $h_m=\widehat h$.

By the finite-point assertion of Lemma~\ref{lem:finite-transport}, using
connectedness of $M$, there is $a\in\GG$ such that $a(h_i(p))\in B$ for $0\leq i\leq m$; if some of the points $h_i(p)$
coincide, they are included only once in the finite set to which the lemma
is applied. Choose an open neighborhood $O$ of $p$, sufficiently small that
$a(h_i(O))\subseteq B$ for $0\leq i\leq m$ and
$h(O)\cap\widehat h(O)=\varnothing$. Choose a nonempty $Y\in\DD$ with
$\overline Y\subseteq O$, small enough that $a(h_i(Y))\Subset B$ for $0\leq i\leq m$.

For $0\leq i<m$, set $S_i=f^i a(h_i(Y))$ and
$R_i=f^i a(h_{i+1}(Y))$. Then
$\overline{S_i},\overline{R_i}\subseteq C_i$. Since the components
$C_0,\ldots,C_{m-1}$ are distinct, the finite joins

\[
Z_s=\bigoplus_{i=0}^{m-1}S_i,
\qquad
Z_t=\bigoplus_{i=0}^{m-1}R_i
\]

belong to $\DD$, and their intersections with $C_i$ are respectively $S_i$
and $R_i$. Consequently $\Dec_{Z_s}(i,h_i(Y))$ and
$\Dec_{Z_t}(i,h_{i+1}(Y))$ for every $i<m$. Moreover,
$h_{i+1}(Y)=s_i(\underline g)(h_i(Y))$, so $\Ad_w(Z_s,Z_t)$ holds.

At $i=0$, the value determined by $Z_s$ is $Y$, while at $i=m-1$ the value
determined by $Z_t$ is $h_m(Y)=\widehat h(Y)$. Thus~\eqref{eq:eval} would
imply $h(Y)=\widehat h(Y)$, contrary to
$h(O)\cap\widehat h(O)=\varnothing$. Hence
$h=\widehat h=w(\underline g)$.
\end{proof}

\begin{cor}\label{cor:member}
For each $r$, define

\[
\Member_r(h,\underline g)
\Longleftrightarrow
\exists w\,
\bigl[
w\text{ is a group word in }r\text{ variables}
\ \wedge\
\Eval_r(w,\underline g,h)
\bigr].
\]

Then $\Member_r(h,\underline g)\Longleftrightarrow
h\in\langle\underline g\rangle$.
\end{cor}

\begin{proof}
Every word code has standard finite length. By
Lemma~\ref{lem:word-define}, $\Eval_r(w,\underline g,h)$ holds exactly when
$h=w(\underline g)$. This proves Theorem~\ref{thm:member}.
\end{proof}

\subsection{Definable local tests for finitely generated subgroups}\label{subsec:local-tests}
Membership makes local movement first-order. If
$L=\langle\underline g\rangle$, then $L$ is locally moving if and only if

\[
\forall\varnothing\neq U\in\DD\ \exists h\in\GG\,
 \bigl[
 \Member_r(h,\underline g)\wedge h\neq1
 \wedge\supp^e(h)\subseteq U
 \bigr].
\]

Indeed, the forward implication is immediate. Conversely, every nonempty
open subset of $M$ contains a nonempty member of $\DD$, so the displayed
condition gives a nontrivial element of $L$ supported in every nonempty open
set.

We next give a first-order condition for approximation on interior collared
open balls. This is weaker than requiring density of the rigid stabilizer in
$\Homeo(M)[O]$ for every arbitrary chart-local open set $O$, as in the
definition in Section~\ref{sec:intro}.

Finite families of elements of $\DD$ contained in one $W\in\DD$ with
$\Loc(W)$ can be coded with a fixed number of variables. Code the empty
family separately by length $m=0$, with all auxiliary support coordinates
empty and group coordinates equal to the identity; there are no decoding
conditions in this case. Now suppose $m\geq1$ and
$A_0,\ldots,A_{m-1}\subseteq W$. Choose a positive string code
$(U,V,W',f,W_0,W_t)$ of size $m$, choose a nonempty $B\ll W_0$, and choose
$a\in\GG$ such that $a(W)\ll B$. Set

\[
Z=\bigoplus_{i=0}^{m-1}f^i a(A_i).
\]

The sets $f^ia(A_i)$ lie in the distinct components
$C_i=f^i(C_0)$ of $V$ and have pairwise separated compact closures. Hence
$Z\in\DD$.

For this purpose we use the formulae $\Slice$ and $\Dec$ from the preceding
subsection without imposing the condition that the relevant slice be
nonempty. Thus, for each $i<m$, the unique slice of $Z$ in $C_i$ is
$f^ia(A_i)$, and applying $a^{-1}f^{-i}$ recovers $A_i$. We require in the
definition of the code that the slice and its decoded value exist as members
of $\DD$ and that the decoded value be contained in $W$. Conversely, every
finite family $A_0,\ldots,A_{m-1}\subseteq W$ admits such a code.

Using two supports $Z_P,Z_Q$ with the same choices of
$(U,V,W',f,W_0,W_t)$ and $a$ therefore codes a finite family
$(P_i,Q_i)_{i<m}$ of pairs of elements of $\DD$ contained in $W$.
Quantification over such pairs $(Z_P,Z_Q)$ is consequently equivalent to
quantification over arbitrary finite families $(P_i,Q_i)_{i<m}$ in $W$.

\begin{prop}\label{prop:ball-approx}
For each fixed $r$ there is a first-order formula
$\on{ball\mbox{-}approx}_r(\underline g)$ which holds exactly when
$L=\langle\underline g\rangle$ has dense rigid stabilizers in
$\Homeo(M)[O]$ for every interior collared open ball $O$.
Every $r$-tuple generating a locally approximating subgroup satisfies
this formula.
\end{prop}

\begin{proof}
Consider the following condition. For all nonempty $A,B,W\in\DD$ satisfying
$A\ll B$, $B\ll W$, and $\Loc(W)$, for every $h\in\GG[A]$, and for every
coded finite family $(P_i,Q_i)_{i<m}$ in $W$ satisfying $P_i\ll Q_i$, there
is $\ell\in\GG$ such that $\Member_r(\ell,\underline g)$,
$\supp^e(\ell)\subseteq B$, and
\begin{equation}\label{eq:local-approx-test}
\ell(P_i)\subseteq h(Q_i)\qquad(i<m).
\end{equation}
By the finite-family coding above, this is a single first-order formula.
For the empty family, the requirements on $\ell$ are satisfied by $\ell=1$.

Suppose first that $L\cap\Homeo(M)[O]$ is dense in $\Homeo(M)[O]$ for every
interior collared open ball $O$. Given $A\ll B\ll W$ as above, choose an
interior collared open ball $O$ with $\cl A\subset O$ and
$\cl O\subset B$. Since $h\in\GG[A]$, we have $h\in\Homeo(M)[O]$.
For a fixed finite family $(P_i,Q_i)_{i<m}$, the conditions
$\ell(\cl P_i)\subseteq h(Q_i)$, wherever they are nontrivial, give a finite
compact-open neighborhood of $h$. Density provides
$\ell\in L\cap\Homeo(M)[O]$ in this neighborhood. Then
$\supp^e(\ell)\subseteq B$ and~\eqref{eq:local-approx-test} holds.

Conversely, assume the first-order condition above, and let $O$ be an
interior collared open ball. It suffices first to approximate homeomorphisms
whose supports have compact closure in $O$. In closed-ball coordinates, an
element of $\Homeo(M)[O]$ can be approximated by conjugating it with radial
rescalings into smaller concentric balls, so these compactly supported
elements are dense in $\Homeo(M)[O]$.

Let $h_0$ have compact support in $O$. Choose $A,B\in\DD$ such that
$\supp^e(h_0)\subseteq A\ll B$ and $\cl B\subset O$, and then choose
$W\in\DD$ with $B\ll W$ and $\Loc(W)$. By local approximation of the
ambient group, we may choose $h\in\GG[A]$ arbitrarily close to $h_0$.

It remains to show that~\eqref{eq:local-approx-test} places $\ell$ in an
arbitrary prescribed finite compact-open neighborhood of $h$. Consider one
of its conditions $\ell(K)\subset T$, where $K$ is compact and $T$ is open,
and assume $h(K)\subset T$. Cover $K\cap\cl B$ by finitely many
$P_i\in\DD$ for which there are $Q_i\in\DD$ satisfying
$P_i\ll Q_i\subseteq W\cap h^{-1}(T)$. Such $P_i,Q_i$ exist by the local
support properties established above. Repeating this construction for each
condition in the finite compact-open neighborhood gives one finite family
$(P_i,Q_i)$.

For the corresponding $\ell$, equation~\eqref{eq:local-approx-test} gives
$\ell(P_i)\subseteq h(Q_i)\subseteq T$. Since the $P_i$ cover
$K\cap\cl B$, we obtain $\ell(K\cap\cl B)\subseteq T$. Outside $\cl B$, both
$\ell$ and $h$ are the identity, and hence $\ell(K\setminus\cl B)\subseteq T$ as
well. Thus $\ell(K)\subseteq T$. Applying this simultaneously to all conditions
in the chosen neighborhood shows that $\ell$ may be taken arbitrarily close to
$h$, and hence arbitrarily close to $h_0$, with $\supp^e(\ell)\subseteq O$.
Therefore $L\cap\Homeo(M)[O]$ is dense in $\Homeo(M)[O]$.
\end{proof}

Proposition~\ref{prop:ball-approx} tests finite intersections of compact-open
conditions while retaining the restriction $\supp^e(\ell)\subseteq B$. It does
not assert density in $\Homeo(M)[U]$ for every chart-local
$U\in\DD$, which may be irregular or have infinitely many components.
Accordingly, we use only the collared-ball version here. This is the form of
local approximation needed in the constructions above.

\subsection{Finiteness properties of elementarily equivalent groups}

Let $H\equiv\GG$, fix $r$, and let $\underline h\in H^r$. Define
$L_{\underline h}=\{x\in H:\Member_r(x,\underline h)\}$. Then
$\langle\underline h\rangle\subseteq L_{\underline h}$, although equality
need not hold.

Indeed, for every fixed literal group word $w$ in $r$ variables, the sentence
$\forall\underline\gamma\,
\Member_r(w(\underline\gamma),\underline\gamma)$ holds in $\GG$, and hence
also in $H$. Therefore every element represented by a standard finite word
in $\underline h$ belongs to $L_{\underline h}$. Moreover, the sentences
expressing $1\in L_{\underline\gamma}$ and closure under
$(x,y)\mapsto x^{-1}y$ hold in $\GG$ and hence in $H$. Thus
$L_{\underline h}$ is a subgroup of $H$ containing
$\langle\underline h\rangle$.
\begin{proof}[Proof of Theorem~\ref{thm:infgen}]
Fix $r$. Since $\GG$ is not generated by any $r$ elements and
$\Member_r(x,\underline\gamma)$ defines
$\langle\underline\gamma\rangle$ in $\GG$, we have

\[
\GG\models
 \forall\underline\gamma\ \exists\delta\
 \neg\Member_r(\delta,\underline\gamma).
\]

By elementary equivalence, the same sentence holds in $H$.

If $\underline h\in H^r$ generated $H$, choose $\delta\in H$ with
$\neg\Member_r(\delta,\underline h)$. But
$\langle\underline h\rangle\subseteq L_{\underline h}$, so
$\delta\notin\langle\underline h\rangle$, contradicting the assumption that
$\underline h$ generates $H$. Hence $H$ is not generated by $r$ elements.
Since $r$ was arbitrary, $H$ is not finitely generated.
\end{proof}

The following is immediate.
\begin{cor}\label{cor:natural-infgen}
No group elementarily equivalent to a non-finitely-generated group admitting
an action in $\goodclass$ is finitely generated. In particular, this applies
to full homeomorphism groups and to the standard one-dimensional full
diffeomorphism and PL groups.
\end{cor}

\subsection{Primality for interpretably presented groups}\label{subsec:prime}
\begin{defn}\label{def:interpretable-presentation}
A finitely generated group $\GG$ admitting an action in $\goodclass$ is \emph{interpretably
presented} if there is a generating tuple $\underline a$ and a parameter-free
formula $\theta(w)$ on the interpreted word sort such that
\[
 \GG\models\theta(w)\quad\Longleftrightarrow\quad w(\underline a)=1
\]
for every standard word $w$ on that tuple. The formula may use the group
structure as well as the arithmetic sort, as it is interpreted in $\GG$. It is part of the hypothesis that
the generating tuple is finite.
\end{defn}

For a fixed arity put
\[
 \Gen(\underline\gamma):=\forall x\ \Member_r(x,\underline\gamma),
\]
and, with $\theta$ as above, set
\begin{equation}\label{eq:chi}
 \chi(\underline\gamma):=
 \Gen(\underline\gamma)\ \wedge\
 \forall w\,[\Eval_r(w,\underline\gamma,1)\leftrightarrow\theta(w)],
\end{equation}
where $w$ is restricted to valid word codes. In $\GG$, whose arithmetic is
standard, this formula defines exactly the $\Aut(\GG)$-orbit of
$\underline a$. Indeed, its realizations generate $\GG$ and have precisely
the same relations and nonrelations as $\underline a$; assigning the
coordinates of one tuple to the other extends to an automorphism.

\begin{thm}\label{thm:prime}
Every finitely generated, interpretably presented group admitting an action
in $\goodclass$ is a prime model of its complete theory.
\end{thm}
\begin{proof}
Let $H\equiv\GG$. Since $\exists\underline\gamma\,\chi(\underline\gamma)$ holds in $\GG$,
choose a tuple $\underline b$ in $H$ satisfying $\chi$. If $\psi$ is any ordinary
group formula true of $\underline a$ in $\GG$, then
\[
 \GG\models\forall\underline\gamma\
               [\chi(\underline\gamma)\Rightarrow\psi(\underline\gamma)]
\]
because $\chi$ is one automorphism orbit. The same sentence holds in $H$,
so $H\models\psi(\underline b)$. Applying this to equations and inequations
between literal words shows that the assignment
$\underline a\mapsto\underline b$ extends to an injective homomorphism
$i:\GG\to H$. Applying it to
$\psi(\underline\gamma)=\varphi(\underline u(\underline\gamma))$ for any
finite tuple of literal words $\underline u$ proves that $i$ is elementary.
This proves primality. Equivalently,~\eqref{eq:chi} isolates the type of a
generating tuple and consequently of every finite tuple of $\GG$.
\end{proof}

Finitely generated recursively presented groups have an arithmetically
definable set of trivial words, and so, when they admit an action in
$\goodclass$, they satisfy Definition~\ref{def:interpretable-presentation}.
More generally, any finitely generated group parameter-free interpretable in
standard arithmetic has an arithmetical word problem for a chosen finite
generating tuple: choose representatives for that tuple in the arithmetic
copy, and express the finite multiplication computation by an arithmetically
coded sequence. The finitely many chosen natural numbers can be named by
numerals. Such a group is therefore interpretably presented as well.

\subsection{Bi-interpretability, QFA, and retractions}\label{subsec:qfa}

We first record a consequence of the word-evaluation formulae above. The
argument is an application of the bi-interpretability criterion of Nies
\cite[Proposition~7.13]{Nies-describing}; see also Khelif
\cite{khelif}. Nies states the criterion using the ring of integers. We use
standard arithmetic instead; the two structures are definably
bi-interpretable.

\begin{prop}\label{prop:bi-interpretable}
Let $\GG$ be a finitely generated group with an action in $\goodclass$.
If $\GG$ is parameter-free interpretable in standard arithmetic, then
$\GG$ is bi-interpretable with standard arithmetic with parameters.
\end{prop}

\begin{proof}
Let $\mathcal N$ be the parameter-free copy of standard arithmetic
interpreted in $\GG$, and suppose conversely that $\GG$ is interpreted in
standard arithmetic. By \cite[Proposition~7.13]{Nies-describing}, it is
enough to define, with parameters, an injection
$\GG\to\mathcal N$.

Fix a generating tuple $\underline a=(a_1,\ldots,a_r)$ for $\GG$. Word
codes in $r$ letters form a definable subset of $\mathcal N$, and
Lemma~\ref{lem:word-define} defines the relation
$\Eval_r(w,\underline a,g)$. For $g\in\GG$, let $\iota(g)$ be the least
word code $w\in\mathcal N$ such that
$\Eval_r(w,\underline a,g)$. Such a code exists because $\underline a$
generates $\GG$.

The graph of $\iota$ is definable in $\GG$ with parameters
$\underline a$. Moreover, $\iota$ is injective: a fixed word in
$\underline a$ evaluates to a unique element of $\GG$, so
$\iota(g)=\iota(h)$ implies $g=h$. Nies's criterion now gives
bi-interpretability of $\GG$ with standard arithmetic, with
$\underline a$ among the allowed parameters.
\end{proof}

\begin{defn}\label{def:qfa}
An infinite finitely generated group $G$ is
\emph{quasi-finitely axiomatizable}, or \emph{QFA}, if there is a
first-order group sentence $\varphi$ such that $G\models\varphi$ and, for
every finitely generated group $H$, $H\models\varphi$ implies $H\cong G$.
\end{defn}

This is the definition of Nies
\cite[Definition~7.1]{Nies-describing}. In particular, the condition that
every finitely generated $H\equiv G$ be isomorphic to $G$ is a consequence
of QFA, not its definition. Likewise, isolating the type of a generating
tuple inside $\Th(G)$ does not by itself produce a sentence which isolates
$G$ among all finitely generated groups.

\begin{cor}\label{cor:qfa}
Under the hypotheses of Proposition~\ref{prop:bi-interpretable}, $\GG$ is
QFA and prime. In particular, this applies to every finitely generated
group with an arithmetical presentation and an action in $\goodclass$.
\end{cor}

\begin{proof}
A finitely generated group with an arithmetical presentation is
interpretable in standard arithmetic: its elements are represented by word
codes, and equality is equality modulo the arithmetically definable word
problem. Hence Proposition~\ref{prop:bi-interpretable} applies.

Primality follows from Theorem~\ref{thm:prime}. For QFA, use the theorem
of Khelif and Nies that a finitely generated structure in a finite
signature which is bi-interpretable with the ring of integers, with
parameters allowed, is QFA
\cite{khelif}; see also \cite[Theorem~7.15]{Nies-describing}. Since standard
arithmetic and the ring of integers are definably bi-interpretable, the
hypotheses are satisfied.
\end{proof}

The QFA conclusion is restricted to the class covered by
Proposition~\ref{prop:bi-interpretable}. In the more general notion of
Definition~\ref{def:interpretable-presentation}, the formula $\theta$
defining the presentation need not be arithmetical. The argument there
still gives primality, but it does not give QFA. For Thompson's groups,
Lasserre \cite{lasserre} proves bi-interpretability with the integers, so
both conclusions hold independently of this distinction.

\begin{prop}\label{prop:retract}
Let $\GG$ be finitely generated and interpretably presented, with an
action in $\goodclass$. If a finitely presented group $H$ satisfies
$H\equiv\GG$, then there are an elementary embedding $i:\GG\to H$ and a
retraction $r:H\to\GG$ such that $r\circ i=\id_{\GG}$.
\end{prop}

\begin{proof}
Let $i:\GG\to H$ be the elementary embedding supplied by
Theorem~\ref{thm:prime}, and let $\underline a$ be a finite generating
tuple for $\GG$. Extend $i(\underline a)$ to a finite generating tuple
$(i(\underline a),\underline b)$ for $H$.

Since $H$ is finitely presented, it has a finite presentation on this
generating tuple. Let
$\underline R(\underline x,\underline y)$ be a finite list of defining
relators. Then
\[
 H\models
 \exists\underline y\,
 \underline R(i(\underline a),\underline y)=1.
\]
By elementarity of $i(\GG)$ in $H$, there is
a tuple $\underline c$ in $\GG$ such that
$\underline R(\underline a,\underline c)=1$.

The assignments $i(\underline a)\mapsto\underline a$ and
$\underline b\mapsto\underline c$ therefore extend to a homomorphism
$r:H\to\GG$. Since $r\circ i$ fixes the generating tuple
$\underline a$, we have $r\circ i=\id_{\GG}$.
\end{proof}

\begin{cor}\label{cor:retraction-iso}
In Proposition~\ref{prop:retract}, $H\cong\GG$ if either
\begin{enumerate}
\item $\GG$ is uniformly simple; or
\item $\GG$ is nonabelian and centerless, while $\GG'$ is uniformly
simple and has bounded commutator width.
\end{enumerate}
\end{cor}

\begin{proof}
Suppose first that $\GG$ is uniformly simple. Uniform simplicity is
expressed by a sentence asserting a fixed bound on the number of conjugates
of $g^{\pm1}$ needed to express any element, uniformly over $g\neq1$.
Hence $H$ is simple. Since $r$ is nontrivial, $\ker r$ is a proper normal
subgroup of $H$, and therefore $\ker r=1$. Thus $r$ is an isomorphism.

For the second case, choose $k$ such that every element of $\GG'$ is a
product of at most $k$ commutators. Let $D_k(x)$ be the formula saying
that $x$ is such a product. In $\GG$, the set defined by $D_k$ is
precisely $\GG'$. The first-order assertions that $D_k$ defines a normal
subgroup and that the corresponding quotient is abelian transfer to $H$.
Since every value of a commutator belongs to $D_k(H)$, it follows that
$D_k(H)=H'$.

Uniform simplicity of $\GG'$ also transfers to the definable subgroup
$H'$. Hence $H'$ is simple. Put $K=\ker r$. Then $K\cap H'$ is normal in
$H'$, so either $K\cap H'=1$ or $H'\subseteq K$. The latter would make
$H/K\cong\GG$ abelian, contrary to the hypothesis on $\GG$. Thus
$K\cap H'=1$. Since $[K,H]\subseteq K\cap H'$, we obtain $[K,H]=1$ and
hence $K\subseteq Z(H)$. Centerlessness transfers from $\GG$ to $H$, so
$K=1$.
\end{proof}

Thus centerlessness and bounded commutator width alone are not the
hypotheses used in the second part of Corollary~\ref{cor:retraction-iso}.
The additional simplicity of the definable derived subgroup is what forces
$K\cap H'$ to be trivial. The relevant simplicity and commutator-width
properties of $F$ and $T$ are supplied by
\cite{GalGis2017,GuelLiousse23}; their QFA properties already follow from
\cite{lasserre}.

\subsection{Turing reducibility}

Let $G$ be a countable group with an arithmetical presentation. Quantifiers
over $G$ can then be replaced uniformly by quantifiers over codes in
$\N$, and the atomic group operations and equality relation are
arithmetical on those codes. By induction on formulae, every group sentence
can therefore be translated effectively into a sentence of first-order
arithmetic. Hence
\[
 \Th(G)\leq_T\Th(\N,+,\times).
\]
This applies in particular to recursively presented groups. In that case
the set of words equal to $1$ is computably enumerable, and therefore
arithmetical; equality of word codes and the induced group operations are
arithmetical as well.

For example, let $\PL_{\Q}([0,1])$ denote the group of increasing
finite-piece PL homeomorphisms of $[0,1]$ whose breakpoints and breakpoint
images are rational. Equivalently, an element is given by finite rational
affine data satisfying the endpoint, continuity, and monotonicity
conditions.

With the stated definition, multiplication, inversion, and equality are
computable from the finite rational data. Thus $\PL_{\Q}([0,1])$ has a
computable presentation and decidable word problem. On the other hand, Theorem~\ref{thm:arith} interprets standard arithmetic
in the group. The group is not finitely generated: either endpoint
slope gives a homomorphism onto the multiplicative group $\Q_{>0}$, which
is not finitely generated.

There is a complementary obstruction for the full homeomorphism groups
considered in \cite{dlNKK22}. Those groups interpret standard second-order
arithmetic. Hence no countable elementary subgroup of one of them can have
a recursive presentation. Indeed, a recursive presentation would imply
that its complete theory is Turing reducible to true first-order
arithmetic, whereas elementarity and the uniform second-order
interpretation would reduce true second-order arithmetic to that same
theory. This is impossible.

\subsection{A one-dimensional counterexample}\label{subsec:counterexample}

We now leave the class $\goodclass$ and construct a countable locally
approximating interval group in which standard arithmetic is not
first-order interpretable, even with parameters.

Consider a language with two sorts: a group sort $P$ and a point sort $D$.
On $D$ we have a dense linear order without endpoints, and $P$ acts
faithfully on $D$ by order-preserving permutations. We write $\mathrm{act}$ as shorthand for the predicate
of the action of $P$ on $D$.

For every standard $k\geq1$, let $\LT_k$ assert the following. Whenever
$a<x_1<\cdots<x_k<b$ and $a<y_1<\cdots<y_k<b$, there is $g\in P$ such
that $g(x_i)=y_i$ for $1\leq i\leq k$ and $g(z)=z$ for every
$z\notin(a,b)$. For each fixed $k$, this is a single first-order sentence.
We also impose
\begin{equation}\label{eq:bounded-support}
 \forall g\in P\ \exists a<b\ \forall z\in D\,
 \bigl(z\leq a\vee z\geq b\Longrightarrow g(z)=z\bigr).
\end{equation}

There is a concrete countable model. Take $D=\Q$ and let $P$ be the group
of compactly supported increasing finite-piece PL homeomorphisms of $\R$
with rational breakpoints and rational breakpoint images. Given the two
ordered $k$-tuples above, interpolate linearly through
$(a,a),(x_1,y_1),\ldots,(x_k,y_k),(b,b)$ and extend by the identity outside
$[a,b]$. The resulting map has rational affine data and verifies $\LT_k$.
Composition and inversion remain in the same group.

We recall the precise form of recursive saturation used below.

\begin{defn}\label{def:recursive-saturation}
Let $\mathcal L$ be a computable countable language and let $\mathcal A$ be
an $\mathcal L$-structure with universe $A$. A \emph{partial type}
$p(\underline x,\underline a)$ is a set of formulae in a fixed finite tuple
$\underline x$ of variables, with parameters from a finite tuple
$\underline a$ in $A$, consistent with $\Th(\mathcal A,\underline a)$.
The type need not be complete. It is \emph{computably enumerable}
(or \emph{c.e.}) if its set of G\"odel codes is computably enumerable in the
language obtained by adding finitely many constant symbols for
$\underline a$.

The type is \emph{finitely satisfiable in $\mathcal A$} if every finite
$p_0\subseteq p$ is satisfied in $\mathcal A$ by some tuple $\underline b$
of the same length as $\underline x$. It is \emph{realized in $\mathcal A$}
if one such tuple satisfies every formula in $p$.

The structure $\mathcal A$ is \emph{recursively saturated} if every c.e.
partial type in finitely many variables, over finitely many parameters
from $A$, that is finitely satisfiable in $\mathcal A$ is realized in
$\mathcal A$.
\end{defn}

This is equivalent to the usual formulation using recursive, rather than
computably enumerable, types, by Craig's trick. We use the c.e. formulation
because the types below arise naturally as effective enumerations. See Barwise and Schlipf \cite{BarwiseSchlipf} and
Schlipf \cite{Schlipf1978}. The following is a standard fact known to logicians, and we include a proof for the reader's convenience.

\begin{lem}\label{lem:recursive-extension}
Every countable structure in a computable countable language has a
countable recursively saturated elementary extension.
\end{lem}

\begin{proof}
Let $\mathcal A_0$ be the given structure. Construct a countable elementary
chain
$\mathcal A_0\prec\mathcal A_1\prec\mathcal A_2\prec\cdots$ as follows.

At stage $j$, consider all computably enumerable types
$p(\underline x,\underline a)$ over finite tuples
$\underline a\in A_j$ which are finitely satisfiable in $\mathcal A_j$.
There are only countably many. Introduce a separate tuple of constants for
each such type and add all these types to the elementary diagram of
$\mathcal A_j$.

Every finite subset of the resulting theory is satisfiable. Indeed, it
involves only finitely many types and finitely many formulae from each
type; the corresponding finite pieces may be realized independently in
$\mathcal A_j$. By compactness, followed by downward
L\"owenheim--Skolem, there is a countable elementary extension
$\mathcal A_{j+1}$ realizing all of these types.

Let $\mathcal A_\omega=\bigcup_j\mathcal A_j$. Suppose
$p(\underline x,\underline a)$ is a c.e. type over a finite tuple from
$A_\omega$ and is finitely satisfiable in $\mathcal A_\omega$. Choose $j$
with $\underline a\in A_j$. Every finite subset $p_0\subseteq p$ gives an
existential formula with parameters in $A_j$ which holds in
$\mathcal A_\omega$. Since $\mathcal A_j\prec\mathcal A_\omega$, that
formula already holds in $\mathcal A_j$. Thus $p$ is finitely satisfiable
in $\mathcal A_j$, and it is realized in $\mathcal A_{j+1}$. Hence
$\mathcal A_\omega$ is recursively saturated.
\end{proof}

Apply Lemma~\ref{lem:recursive-extension} to the two-sorted rational action
above, and write the resulting elementary extension as
$(G,D',<,\mathrm{act})$. The point sort $D'$ is again a countable dense
linear order without endpoints, hence $D'\cong\Q$. Its Dedekind completion
is order-isomorphic to $\R$, and adjoining its two endpoints gives a
compact interval. Every order automorphism of $D'$ extends uniquely first
to its Dedekind completion and then to the two endpoints. Thus, after
identifying the compactification with $[0,1]$, we obtain a faithful action
$G\leq\Homeo_+([0,1])$.
Faithfulness follows from faithfulness on the dense point sort. By
\eqref{eq:bounded-support}, every $g\in G$ fixes two tails of $D'$ and
therefore has compact support in $(0,1)$ after completion.

\begin{lem}\label{lem:ordered-density}
The action $G\leq\Homeo_+([0,1])$ is locally approximating, including for
disconnected chart-local open sets occurring in the definition in
Section~\ref{sec:intro}.
\end{lem}

\begin{proof}
Let $U$ be such an open set and let $h\in\Homeo([0,1])[U]$. Fix
$\varepsilon>0$. Since $h$ is increasing, choose finitely many points of
$D'$ whose consecutive $h$-images, including the two end intervals, have
diameter less than $\varepsilon/3$.

For every chosen point lying in $U$, choose a point of $D'$ sufficiently
close to its $h$-image and in the same component of $U$, preserving the
order of the finite list. For a chosen point outside $U$, take the point
itself as its target.

Only finitely many components of $U$ contain points which must be moved.
For each such component, choose $a,b\in D'$ inside that component and
enclosing all corresponding source and target points. An appropriate
sentence $\LT_k$ gives an element of $G$ realizing those assignments and
fixing $D'\setminus(a,b)$. By density of $D'$ and continuity of the
extended action, the resulting homeomorphism fixes the whole complement
of $(a,b)$. Taking the finite product over the relevant components gives
an element of $G[U]$ realizing all prescribed finite assignments.

Monotonicity bounds the image of every point between the images of its
neighboring test points. If the target points were chosen sufficiently
close to the corresponding $h$-images, this gives uniform
$\varepsilon$-approximation to $h$. Applying the same estimate to the
inverse, or using uniform continuity of $h^{-1}$, gives approximation in
the usual homeomorphism-group topology. The argument is unchanged for a
boundary chart, since all points which are moved lie in the interior and
each interpolation is compactly supported in $U$.
\end{proof}

The group reduct of $(G,D',<,\mathrm{act})$ is recursively saturated:
every c.e. type in the group language is also a c.e. type in the expanded
language. We now use the basic obstruction supplied by recursive
saturation, which is essentially a version of the Compactness Theorem.

\begin{lem}\label{lem:no-standard-arithmetic}
Let $\mathcal A$ be recursively saturated.
\begin{enumerate}
\item No fixed finite-dimensional first-order interpretation in
$\mathcal A$, even with finitely many parameters, has quotient structure
isomorphic to standard arithmetic.
\item If $G$ is a recursively saturated group, no infinite finitely
generated subgroup of $G$ is definable, even with additional finitely many
parameters.
\end{enumerate}
\end{lem}

\begin{proof}
For the first assertion, suppose an interpretation of arithmetic is given
by a definable domain
$\delta(\underline x,\underline p)$, a definable equivalence relation
$E$, and formulae defining the interpreted zero and successor. The tuple
length of $\underline x$ is fixed.

For each standard $n\in\N$, one can effectively construct a formula
$\nu_n(\underline x,\underline p)$ asserting that the $E$-class of
$\underline x$ is the $n$th successor of zero. Consider the c.e. type $\{\delta(\underline
x,\underline p)\} \cup \{\neg\nu_n(\underline x,\underline p):n\in\N\}$. Every finite subset is
satisfiable: in the interpreted standard copy of
$\N$, choose a numeral larger than all the finitely many excluded ones.
Recursive saturation therefore gives a tuple in the interpreted domain
whose $E$-class is different from every standard numeral. This contradicts
the assumption that the interpreted quotient is isomorphic to standard
$\N$.

For the second assertion, suppose that
$\theta(x,\underline p)$ defines the infinite subgroup
$\langle\underline a\rangle$. Consider $\{\theta(x,\underline p)\} \cup \{x\neq w(\underline a):
w\text{ is a literal finite group word}\}$. This is a c.e. type. Every finite subset is satisfiable
because
$\langle\underline a\rangle$ is infinite, whereas a finite subset excludes
only finitely many of its elements. A realization would belong to
$\langle\underline a\rangle$ without being equal to any word in
$\underline a$, a contradiction.
\end{proof}

In particular, no nontrivial cyclic subgroup of the group constructed
above is definable: every nonidentity increasing homeomorphism of an
interval has infinite order.

\begin{cor}\label{cor:negative}
There is a countable group $G\leq\Homeo_+([0,1])$ that is locally
approximating and in which standard arithmetic is not first-order
interpretable, even with parameters.
\end{cor}

\begin{proof}
Take the group $G$ constructed above. Local approximation is
Lemma~\ref{lem:ordered-density}, while the group reduct is recursively
saturated. Apply Lemma~\ref{lem:no-standard-arithmetic}.
\end{proof}

For the circle example, begin instead with the two-sorted structure
$(P,D,C,\mathrm{act})$, where $D=\Q/\Z$ has its usual cyclic order $C$
and $P$ is the countable group of orientation-preserving finite-piece
PL homeomorphisms of $\R/\Z$ with rational breakpoints and rational
breakpoint images. For each standard $k\geq1$, this structure satisfies
the cyclic-order analogue of $\mathrm{LT}_k$: two ordered $k$-tuples
in a proper oriented arc with endpoints in $D$ can be matched by an
element fixing the complement of that arc pointwise. This follows by
rational increasing PL interpolation in a lift of the arc.

Take a countable recursively saturated elementary extension
$(G_\circ,D_\circ,C_\circ,\mathrm{act})$. The point sort is again a
countable dense cyclic order, hence isomorphic to $\Q/\Z$.
Its cyclic-order completion is a circle, and every cyclic-order
automorphism extends uniquely to an orientation-preserving
homeomorphism of that circle. Thus the group acts faithfully on $S^1$.

To verify local approximation, let $U$ have compact closure in a circle
chart and let $h\in\Homeo(S^1)[U]$. Choose a point of $D_\circ$ outside
$\overline U$ and cut the circle there. The finite-mesh argument of
Lemma~\ref{lem:ordered-density} now applies in the resulting interval.
Only finitely many components of $U$ need be used, and in each such
component the cyclic-order interpolation sentences provide the required
element with support in a smaller arc. Their finite product belongs to
$G_\circ[U]$ and approximates $h$ in the compact-open topology.
The group reduct is recursively saturated, so
Lemma~\ref{lem:no-standard-arithmetic} excludes both an interpretation
of standard arithmetic and definability of any infinite finitely
generated subgroup, even with parameters. This proves the circle
assertion of Theorem~\ref{thm:negative-intro}.

\subsubsection{A group elementarily equivalent to $F$}\label{subsec:F-counterexample}

Consider the standard ordered action $(F,\Z[1/2]\cap(0,1),<,\mathrm{act})$. It satisfies every
sentence $\LT_k$. Indeed, let the source and target
tuples consist of increasing dyadic points in a dyadic interval $(a,b)$.
Decompose each interval between consecutive source points into finitely
many standard dyadic intervals, and do the same for the corresponding
target interval. By repeated bisection, the two decompositions may be
refined to have the same number of pieces. Mapping corresponding pieces
affinely gives slopes which are powers of two. Extending by the identity
outside $[a,b]$ therefore gives an element of $F$ realizing the prescribed
finite correspondence.

Take a countable recursively saturated elementary extension of this
\emph{expanded action}. Let $H$ be its group reduct. Then $H\equiv F$,
and the completion argument above gives a locally approximating action of
$H$ on an interval. The bounded-support sentence
\eqref{eq:bounded-support} is not needed here.

The group $H$ cannot be finitely generated. More generally, no infinite
recursively saturated group is finitely generated. For a finite tuple
$\underline h$, the c.e. type
$\{x\neq w(\underline h):w\text{ a literal group word}\}$ is finitely
satisfiable; if $\underline h$ generated the whole group, recursive
saturation would realize this type and give an element not represented by
any word in $\underline h$. Thus $H$ is an infinitely generated group
elementarily equivalent to $F$ and admitting a locally approximating
interval action.

There is no contradiction with QFA of $F$, since QFA only distinguishes $F$
among finitely generated groups; nor is the new
topological action an elementary extension of the original action on its
completed interval. The resulting homeomorphisms need not be finite-piece
PL, and the action need not satisfy component selection. Likewise, a fixed
interpretation which gives standard arithmetic in $F$ produces in $H$ a
model elementarily equivalent to standard arithmetic, but recursive
saturation forces that interpreted model to be nonstandard.

It is also useful to separate definability from abstract subgroup
structure. Let $f$ be a nontrivial increasing homeomorphism. After
replacing $f$ by $f^{-1}$ if necessary, choose $x$ with $x<f(x)$ and a
nontrivial $g$ supported in an interval $I\Subset(x,f(x))$. Then
$f^n(I)\subset(f^n(x),f^{n+1}(x))$ for every $n\in\Z$, so these intervals
are pairwise disjoint. Consequently the conjugates
$f^ngf^{-n}$ commute and generate a direct sum of infinite cyclic groups,
while conjugation by $f$ shifts the summands. Hence
$\langle f,g\rangle\cong\Z\wr\Z$.

This applies to the recursively saturated counterexample as well. Thus
the obstruction is not the absence of finitely generated subgroups in
which arithmetic can be interpreted abstractly. The obstruction is their
definability in the ambient group. In particular, the interpretation of
arithmetic in the wreath product from \cite{AM2009} cannot be transferred
to an arbitrary overgroup without a definable way to recover that
subgroup.

\section{Action rigidity of locally approximating groups}\label{sec:rigidity}

In this section, we prove action rigidity for locally approximating groups. This part of the paper does not
rely on an interpretation of arithmetic.

\subsection{Finite nerve towers and dimension}
The construction below is a finite controlled version of the classical nerve
theorem and its proof by carriers. Nerve theorems in various forms go back to
Leray, Borsuk, and Weil; in particular, Weil's homotopy theorem constructs a
homotopy equivalence between a space and the nerve of a suitable open cover
\cite{Leray1945,Borsuk1948,Weil1952}. The formulation of the argument
also keeps the resulting maps and homotopies subordinate to prescribed
members of the cover; see, for example, \cite[Corollaries~2.1 and~2.2]{Nagorko2007}.

Our situation is slightly different. We do not assume that all nonempty
intersections at a single level form a good cover. Instead, the successive
levels of the finite tower provide the extensions needed in the usual
simplicial and prism constructions. We keep these extensions explicit because
the resulting control on the homotopy tracks will be needed below.

Let $K$ be a nonempty finite simplicial complex, with vertex set $A$ and
dimension $d$. For an integer $R\geq0$, a \emph{ball-controlled tower with
nerve $K$} is a family $\UU=(U_a^j)_{a\in A,\ 0\leq j\leq R}$ of open sets such that, writing
$U_\sigma^j=\bigcap_{a\in\sigma}U_a^j$, we have:
\begin{enumerate}
\item\label{tower:nerve} at each level $j$, $U_\sigma^j\ne\varnothing$ if and only if the
nonempty subset $\sigma\subseteq A$ is a simplex of $K$;
\item for every nonempty simplex $\sigma$ and every $j<R$,
\begin{equation}\label{eq:tower-inclusions}
 U_\sigma^j\ll U_\sigma^{j+1}.
\end{equation}
\end{enumerate}
A \emph{support tower} is a ball-controlled tower with every $U_a^j$ in
$\DD$. It is \emph{covering} if $\bigcup_aU_a^0$ is dense in the ambient
compact manifold, equivalently if no nonempty member of $\DD$ is disjoint
from all $U_a^0$. For fixed $K,R$, these are single first-order conditions
on support towers. Their intersections need not themselves belong to $\DD$.
We refer to Condition~\ref{tower:nerve} as the \emph{nerve condition};
its instance at level $R$ is the same condition for $(U_a^R)_a$.

We begin with a finite nerve lemma in which only one direction of the usual
nerve conclusion is needed. If $\sigma$ is a nonempty simplex of $K$, define its
closed star by

\[
\overline{\St}_K(\sigma)
 =
 \bigcup\{|\eta|:\eta\cup\sigma\in K\}.
\]

This set is contractible: on each simplex $|\eta\cup\sigma|$ occurring in
the union, contract linearly to the barycenter of $\sigma$. These
contractions agree on common faces. We will also use the inclusions

\[
\sigma\subseteq\tau
 \quad\Longrightarrow\quad
 \overline{\St}_K(\tau)\subseteq\overline{\St}_K(\sigma).
\]

\begin{lem}[One-sided finite nerve lemma]\label{lem:one-sided-nerve}
Suppose $R\geq d$ and $\UU$ is a ball-controlled tower in a manifold, where
$d=\dim K$. Put $O=\bigcup_a U_a^R$. Then there are continuous maps
$q:|K|\to O$ and $p:O\to|K|$ such that
$pq\simeq\id_{|K|}$. Consequently,
$H_*(|K|;\Ftwo)$ is a direct summand of $H_*(O;\Ftwo)$.
\end{lem}

\begin{proof}
We define $q$ on the barycentric subdivision $\sd K$. A vertex of $\sd K$
corresponds to a nonempty simplex $\sigma$ of $K$; choose
$x_\sigma\in U_\sigma^0$ and send that vertex to $x_\sigma$.

A simplex of $\sd K$ has the form
$\Delta=[\sigma_0\subsetneq\cdots\subsetneq\sigma_m]$.

We construct $q$ by induction on $m$, requiring
\begin{equation}\label{eq:q-carrier}
q(|\Delta|)\subseteq U_{\sigma_0}^{m}.
\end{equation}
For $m=0$ this is our choice of $x_{\sigma_0}$. Suppose $m>0$ and $q$ has
already been defined on $\partial\Delta$. Every face of $\Delta$ is mapped
into $U_{\sigma_0}^{m-1}$. Indeed, a face which retains $\sigma_0$ has the
same least simplex, while the face obtained by deleting $\sigma_0$ has least
simplex $\sigma_1\supseteq\sigma_0$, and the nesting of the tower gives the
same inclusion.

By~\eqref{eq:tower-inclusions}, the inclusion
$U_{\sigma_0}^{m-1}\subseteq U_{\sigma_0}^{m}$ factors through a collared
open ball. Hence the already defined map
$\partial\Delta\to U_{\sigma_0}^{m-1}$ extends over $\Delta$ by coning
inside that ball. This proves~\eqref{eq:q-carrier} for every simplex of
$\sd K$. Since $m\leq d\leq R$, the image of $q$ lies in $O$.

Choose a partition of unity $(\lambda_a)_a$ on $O$ subordinate to the finite
cover $(U_a^R)_a$, and define
$p(x)=\sum_a\lambda_a(x)a\in|K|$. The nerve condition guarantees that
this barycentric sum lies in $|K|$.

Let $x\in|\Delta|$, where
$\Delta=[\sigma_0\subsetneq\cdots\subsetneq\sigma_m]$. By
\eqref{eq:q-carrier}, $q(x)\in U_{\sigma_0}^m\subseteq U_{\sigma_0}^R$.
If $\lambda_a(q(x))>0$, then also $q(x)\in U_a^R$. If $\eta$ is the simplex
spanned by all such vertices $a$, the nerve condition at level $R$
therefore gives $\eta\cup\sigma_0\in K$. Hence
$pq(x)\in\overline{\St}_K(\sigma_0)$.

The inclusion $|\Delta|\hookrightarrow|K|$ also has image in
$\overline{\St}_K(\sigma_0)$. Thus, on each $\Delta$, both $pq$ and the
identity take values in the same contractible closed star. We construct a
homotopy from $pq$ to $\id_{|K|}$ inductively over the simplices of $\sd K$.
For a vertex $v_\sigma$ of $\sd K$, both $pq(v_\sigma)$ and
$v_\sigma$, the barycenter of $\sigma$, lie in
$\overline{\St}_K(\sigma)$. Join them by the straight-line path to
$v_\sigma$ inside $\overline{\St}_K(\sigma)$. This defines the homotopy on
the $0$-skeleton.
For a simplex $\Delta$ of positive dimension, suppose that the homotopy
has been defined on its boundary. For a
face whose least simplex is $\tau\supseteq\sigma_0$, its image lies in
$\overline{\St}_K(\tau)\subseteq\overline{\St}_K(\sigma_0)$. The boundary
homotopy therefore extends across $\Delta\times[0,1]$ inside
$\overline{\St}_K(\sigma_0)$. Since $K$ is finite, these homotopies paste
to give $pq\simeq\id_{|K|}$.

It follows on homology that $p_*q_*=\id$. Hence $q_*$ is injective and
$H_*(|K|;\Ftwo)$ is a direct summand of $H_*(O;\Ftwo)$.
\end{proof}

\begin{prop}\label{prop:dimension}
For every $k\geq2$, there is a first-order sentence
$\mathrm{dim}_{\geq k}$ which, among locally approximating manifold
actions, holds exactly in dimensions at least $k$.
\end{prop}

\begin{proof}
Fix $k\geq2$, let $K=\partial\Delta^k$, and put $R=k-1=\dim K$. The sentence
$\mathrm{dim}_{\geq k}$ asserts the existence of a ball-controlled tower
with nerve $K$, all of whose sets are contained in one $E\in\DD$ with
$\Loc(E)$. In particular, the tower is required to lie in one interior
coordinate neighborhood; it is not required to cover $M$.

Suppose first that the sentence holds and $\dim M=m$. Let
$O=\bigcup_aU_a^R$. Since all $U_a^R$ lie in $E$, the set $O$ is
homeomorphic to an open subset of $\R^m$. Lemma~\ref{lem:one-sided-nerve}
gives an injection

\[
H_{k-1}(S^{k-1};\Ftwo)
 \longrightarrow
 H_{k-1}(O;\Ftwo).
\]

Since $H_{k-1}(S^{k-1};\Ftwo)\cong\Ftwo$, we have
$H_{k-1}(O;\Ftwo)\neq0$.

On the other hand, an open subset of $\R^m$ has zero homology in every
degree $\geq m$. For degrees $>m$ this is the usual dimension vanishing,
while in degree $m$ it follows from the vanishing of top homology of each
noncompact $m$-manifold component. Equivalently, one may obtain both
statements from manifold duality with compact supports; see
\cite[Section~3.3]{Hatcher}. Hence $k-1<m$, and therefore $m\geq k$.

Conversely, suppose $m\geq k$. Inside one interior coordinate ball, embed
a standard $S^{k-1}$ with a trivial normal neighborhood
$S^{k-1}\times B^c(1)$, where $c=m-k+1$ and
$B^c(r)=\{y\in\R^c:\|y\|<r\}$. Give the sphere the combinatorial
triangulation $K=\partial\Delta^k$, with barycentric coordinates
$(\lambda_a)_a$. Choose
\[
\begin{gathered}
0<\varepsilon_R^-<\varepsilon_R^+<\cdots<
\varepsilon_0^-<\varepsilon_0^+<1/k,\\
0<r_0^-<r_0^+<r_1^-<r_1^+<\cdots<r_R^-<r_R^+<1,
\end{gathered}
\]
and set
\[
 P_a^j=\{\lambda_a>\varepsilon_j^+\}\times B^c(r_j^-),\qquad
 Q_a^j=\{\lambda_a>\varepsilon_j^-\}\times B^c(r_j^+).
\]
These are truncated \emph{open} vertex stars, thickened in the normal
directions. Both systems have nerve $K$: a face barycenter witnesses each
simplex intersection, and the nonzero coordinates of any point span a
simplex. Every nonempty intersection is an open $m$-ball. Indeed,
subtracting the threshold from the coordinates in $\sigma$ and rescaling
identifies the sphere factor with the open star of $\sigma$ in the
combinatorial sphere; taking its product with $B^c(r)$ gives the asserted
ball. This coordinate change is written out in
Lemma~\ref{lem:fine-towers}.

The parameters give $\cl P_a^j\subseteq Q_a^j$ and
$\bigcap_{a\in\sigma}\cl Q_a^j\subseteq P_\sigma^{j+1}$ for $j<R$.
By Lemma~\ref{lem:large-subset}, choose $U_a^j\in\DD$ with
$\cl P_a^j\subseteq U_a^j$ and $\cl U_a^j\subseteq Q_a^j$.
The resulting families have nerve $K$. Moreover, $\cl U_\sigma^j$ is
compact and lies in the open ball $P_\sigma^{j+1}\subseteq U_\sigma^{j+1}$;
a collared closed ball inside $P_\sigma^{j+1}$ therefore witnesses
$U_\sigma^j\ll U_\sigma^{j+1}$. Finally, the entire construction has
compact closure in the original coordinate ball. A support envelope
contains all its sets in one $E\in\DD$ with $\Loc(E)$. This proves
$\mathrm{dim}_{\geq k}$.

Here the standing assumption $\dim M>0$ is used explicitly: set
$\mathrm{dim}_{\geq1}$ to be the tautologically true sentence, and define
$\mathrm{dim}_m$ by
$\mathrm{dim}_{\geq m}\wedge\neg\mathrm{dim}_{\geq m+1}$. These sentences
therefore recover the dimension of the underlying manifold.
\end{proof}

\subsection{Closedness and one-manifolds}\label{subsec:closedness}

For each fixed $s\geq1$, the existence of $V_1,\ldots,V_s\in\DD$ with
$\Loc(V_i)$ and dense union is first-order. Indeed, density of
$\bigcup_{i=1}^sV_i$ is equivalent to

\[
\forall U\in\DD\,
 \bigl[
 U\neq\varnothing
 \Longrightarrow
 \bigvee_{i=1}^s U\cap V_i\neq\varnothing
 \bigr].
\]

Thus, for each fixed $s$, this gives a single first-order sentence.

Suppose $M$ is compact with nonempty boundary. Every $V_i\in\DD$ satisfying $\Loc(V_i)$ has
$\overline{V_i}$ compactly contained in $\Int M$. Hence
$\bigcup_{i=1}^s\overline{V_i}$ is a compact subset of $\Int M$ and cannot
be dense in $M$. Therefore no finite family of local supports has dense
union.

Now suppose $M$ is closed. Choose finitely many interior collared open balls
whose slightly smaller concentric balls still cover $M$. Applying
Lemma~\ref{lem:large-subset} between the smaller and larger balls gives finitely many
$V_1,\ldots,V_s\in\DD$, each satisfying $\Loc(V_i)$, whose union is all of
$M$, and in particular is dense.

It follows that if the groups in two locally approximating actions are elementarily
equivalent, then their manifolds are either both closed or both have
nonempty boundary. Note that no uniform value of $s$ over all closed manifolds is
required.

The closedness argument itself does not use connectedness. Together with
Proposition~\ref{prop:dimension}, it gives action rigidity for compact
connected one-manifolds; connectedness is used in the classification that follows. If $\dim M=1$, then a compact
connected manifold is homeomorphic to $S^1$ when it is closed and to
$[0,1]$ when it has nonempty boundary. Thus elementary equivalence of the
corresponding locally approximating actions determines the homeomorphism
type of the underlying one-manifold.

\subsection{The controlled finite nerve theorem}\label{subsec:controlled-nerve}

We now strengthen Lemma~\ref{lem:one-sided-nerve}. The same finite nerve
will give homotopy equivalences, with control on the homotopies.

Fix a finite simplicial complex $K$ of dimension $d$, and let $\UU$ be a
covering tower of depth
\begin{equation}\label{eq:depth}
R\geq d+2
\end{equation}
in a compact manifold $X$. Since $\bigcup_aU_a^0$ is dense and
$\overline{U_a^0}\subseteq U_a^1$ for every $a$, finiteness gives
$X=\bigcup_a\overline{U_a^0}\subseteq\bigcup_aU_a^1$. Thus the sets
$U_a^1$ cover $X$. Notice also that this forces $X$ to be closed: the
finitely many sets $\overline{U_a^0}$ are compact subsets of $\Int X$ and
their union is $X$.

Choose a partition of unity $(\phi_a)_a$ on $X$ such that
\begin{equation}\label{eq:closed-supports}
F_a:=\supp(\phi_a)\subseteq U_a^1.
\end{equation}
Here $\supp(\phi_a)$ denotes the ordinary closed support. Such a partition
exists by first taking a finite closed shrinking of the cover
$(U_a^1)_a$. Define $p:X\to|K|$ by
$p(x)=\sum_a\phi_a(x)a$, and let $q:|K|\to X$ be the map constructed in
Lemma~\ref{lem:one-sided-nerve}. For $z\in|K|$, write $\suppv(z)$ for the
set of vertices having nonzero barycentric coordinate in $z$. This is a
nonempty simplex of $K$.

\begin{prop}[Finite controlled nerve theorem]\label{prop:controlled-nerve}
The maps $p:X\to|K|$ and $q:|K|\to X$ are homotopy inverses. They admit
homotopies $L:pq\simeq\id_{|K|}$ and
$H_X:qp\simeq\id_X$ with the following properties.
\begin{enumerate}
\item If
$\Delta=[\sigma_0\subsetneq\cdots\subsetneq\sigma_m]$ is a simplex of
$\sd K$, then $q(\Delta)\subseteq U_{\sigma_0}^m$ and
$L(\Delta\times[0,1])\subseteq\overline{\St}_K(\sigma_0)$.

\item For every $x\in X$, there is a nonempty simplex
$\tau\subseteq\suppv(p(x))$ such that
$H_X(\{x\}\times[0,1])\subseteq U_\tau^R$. Moreover
$p(H_X(x,t))\in\overline{\St}_K(\tau)$ for every $t\in[0,1]$.
\end{enumerate}
\end{prop}

\begin{proof}
The map $q$ is constructed exactly as in
Lemma~\ref{lem:one-sided-nerve}. In particular, if
$\Delta=[\sigma_0\subsetneq\cdots\subsetneq\sigma_m]$, then
$q(\Delta)\subseteq U_{\sigma_0}^m$.

The same proof gives the homotopy $L$. Indeed, suppose $z\in\Delta$ and
$\phi_a(q(z))>0$. Then $q(z)\in F_a\subseteq U_a^1\subseteq U_a^R$.
Also $q(z)\in U_{\sigma_0}^m\subseteq U_{\sigma_0}^R$. Hence, if
$\eta=\suppv(pq(z))$, the nerve condition at level $R$ gives
$\eta\cup\sigma_0\in K$. Therefore
$pq(z)\in\overline{\St}_K(\sigma_0)$. The identity map on $\Delta$ also
takes values in this closed star. The inductive prism construction from
Lemma~\ref{lem:one-sided-nerve} consequently gives
$L(\Delta\times[0,1])\subseteq\overline{\St}_K(\sigma_0)$.

It remains to construct $H_X$. For $\sigma\in K$, put
$F_\sigma=\bigcap_{a\in\sigma}F_a$, and define
\begin{equation}\label{eq:resolution}
E_F=\bigcup_{\sigma\in K}F_\sigma\times|\sigma|
\subseteq X\times|K|.
\end{equation}
The map $s:X\to E_F$ given by $s(x)=(x,p(x))$ is a continuous section of
the first projection. Indeed, if $\rho=\suppv(p(x))$, then
$x\in F_\rho$ and $p(x)\in|\rho|$.

The compact space $E_F$ has a finite closed cover by the sets
\begin{equation}\label{eq:resolution-pieces}
F_{\sigma_m}\times|\Delta|,
\qquad
\Delta=[\sigma_0\subsetneq\cdots\subsetneq\sigma_m]\in\sd K.
\end{equation}
To see this, take $(x,z)\in E_F$ and let $\Delta$ be the smallest simplex
of $\sd K$ containing $z$. Its maximal simplex is
$\sigma_m=\suppv(z)$, and $x\in F_{\sigma_m}$. Intersections of the sets
in~\eqref{eq:resolution-pieces} occur along the corresponding faces. If
the maximal simplex is deleted from a chain, the associated $F$-set only
gets larger.

We construct a continuous map
$G:E_F\times[0,1]\to X$ such that $G(x,z,0)=x$ and
$G(x,z,1)=q(z)$. The construction is inductive over the simplices of
$\sd K$.

For a vertex $[\sigma]$ of $\sd K$, both $x\in F_\sigma\subseteq U_\sigma^1$
and $q([\sigma])\in U_\sigma^0\subseteq U_\sigma^1$. By the tower
condition, the inclusion $U_\sigma^1\subseteq U_\sigma^2$ factors through
a collared ball. Fix a contraction of that ball to one point. Contract
$x$ to this point and then follow the reverse contraction from
$q([\sigma])$. This defines $G$ continuously on
$F_\sigma\times\{[\sigma]\}\times[0,1]$, with image in $U_\sigma^2$.

Now let
$\Delta=[\sigma_0\subsetneq\cdots\subsetneq\sigma_m]$, with $m>0$, and
suppose $G$ has been defined over all proper faces of $\Delta$. On
$F_{\sigma_m}\times\partial(\Delta\times[0,1])$ the map is already
determined: at time $0$ it is $x$, at time $1$ it is $q(z)$, and on the
side faces it is given by the induction hypothesis. Its entire image lies
in $U_{\sigma_0}^{m+1}$. For the two end faces, this follows from
$F_{\sigma_m}\subseteq U_{\sigma_m}^1\subseteq U_{\sigma_0}^{m+1}$ and
$q(\Delta)\subseteq U_{\sigma_0}^m$. For a side face, the induction
hypothesis gives the same bound; if the first simplex of the chain has
been deleted, use
$U_{\sigma_1}^{m+1}\subseteq U_{\sigma_0}^{m+1}$.

By the definition of the ball-controlled tower, the inclusion
$U_{\sigma_0}^{m+1}\subseteq U_{\sigma_0}^{m+2}$ factors through a
collared ball. Identify the prism $\Delta\times[0,1]$ with the cone on its
boundary. Using a fixed contraction of this ball, extend the boundary map
over the cone. Because the contraction is fixed, this construction depends
continuously on $x\in F_{\sigma_m}$, and the value at the cone point is
independent of the boundary variable. Thus $G$ extends over
$F_{\sigma_m}\times\Delta\times[0,1]$ with image in
$U_{\sigma_0}^{m+2}$.

The extensions agree on common faces. Since
\eqref{eq:resolution-pieces} is a finite closed cover, the pasting lemma
gives a continuous homotopy $G$ on all of $E_F$.

Now define $H_X$ by restricting $G$ to the section $s$ and reversing the
time parameter. Thus $H_X$ is a homotopy from $qp$ to $\id_X$.

Fix $x\in X$, and let
$\Delta=[\sigma_0\subsetneq\cdots\subsetneq\sigma_m]$ be the smallest
barycentric simplex containing $p(x)$. Put
$\rho=\sigma_m=\suppv(p(x))$ and $\tau=\sigma_0$. Then
$\tau\subseteq\rho$, and the preceding construction gives
$H_X(\{x\}\times[0,1])\subseteq U_\tau^{m+2}\subseteq U_\tau^R$, since
$m\leq d$.

Finally, let $y=H_X(x,t)$ and put $\eta=\suppv(p(y))$. If $a\in\eta$,
then $\phi_a(y)>0$, so $y\in F_a\subseteq U_a^1\subseteq U_a^R$. Since
also $y\in U_\tau^R$, the nerve condition at level $R$ implies
$\eta\cup\tau\in K$. Hence
$p(y)\in\overline{\St}_K(\tau)$. This proves (2) and completes the proof.
\end{proof}

We next record the amount of control needed for comparison between two
manifolds. For a nonempty simplex $\sigma$ of $K$, let $N_j(\sigma)$ be
the set of vertices whose distance in the $1$-skeleton of $K$ from some
vertex of $\sigma$ is at most $j$. Let $\alpha$ be an open cover of $X$.
We say that the tower on $X$ is \emph{$\alpha$-small} if, for every
nonempty simplex $\sigma$,
\begin{equation}\label{eq:alpha-small}
\bigcup_{a\in N_6(\sigma)}U_a^R
\quad\text{is contained in some }A_\sigma\in\alpha.
\end{equation}
This is an external condition on the particular tower chosen in $X$, and in particular
is not part of the first-order sentence which will be transferred to
another manifold.

Recall that a map $f:Y\to X$ is an \emph{$\alpha$-equivalence} if there is
a homotopy inverse $g:X\to Y$ such that every track of a homotopy
$fg\simeq\id_X$ is contained in one member of $\alpha$, and, for every
track of a homotopy $gf\simeq\id_Y$, its image under $f$ is contained in
one member of $\alpha$. The member of $\alpha$ may depend on the initial
point.

\begin{prop}[Controlled comparison]\label{prop:controlled-comparison}
Let $X$ and $Y$ be compact manifolds carrying covering towers of depth
$R\geq d+2$ with the same finite nerve $K$. If the tower on $X$ is
$\alpha$-small, then there is an $\alpha$-equivalence $Y\to X$.
\end{prop}

\begin{proof}
Apply Proposition~\ref{prop:controlled-nerve} to the two towers, and write
$p_X,q_X,L_X,H_X$ and $p_Y,q_Y,L_Y,H_Y$ for the resulting maps and
homotopies. Define $f=q_Xp_Y:Y\to X$ and $g=q_Yp_X:X\to Y$.

We first record three containment estimates. For $Z=X$ or $Y$ and
$z\in|K|$, the point $q_Z(z)$ belongs to $U_{Z,a}^R$ for some
$a\in\suppv(z)$. Consequently
$\suppv(p_Zq_Z(z))\subseteq N_1(\suppv(z))$.

Property (1) of Proposition~\ref{prop:controlled-nerve} similarly gives
$\suppv(L_Z(z,t))\subseteq N_1(\suppv(z))$, while property (2) gives
$\suppv(p_ZH_Z(y,t))\subseteq N_1(\suppv(p_Z(y)))$. Indeed, in each case
a vertex which occurs in the new support belongs to a simplex together
with a nonempty face of the original support.

We first control a homotopy from $fg$ to $\id_X$. Concatenate
$q_XL_Yp_X$, which joins $fg$ to $q_Xp_X$, with $H_X$. Fix $x\in X$ and
put $\rho=\suppv(p_X(x))$. Along the first part,
$\suppv(L_Y(p_X(x),t))\subseteq N_1(\rho)$, so its image under $q_X$ lies
in $\bigcup_{a\in N_1(\rho)}U_{X,a}^R$. Along the second part, property
(2) gives a nonempty $\tau\subseteq\rho$ such that the entire track lies
in $U_{X,\tau}^R$. Hence the complete track is contained in
$\bigcup_{a\in N_1(\rho)}U_{X,a}^R$, and therefore in one member of
$\alpha$ by~\eqref{eq:alpha-small}.

Now consider the homotopy from $gf$ to $\id_Y$ obtained by concatenating
$q_YL_Xp_Y$ with $H_Y$, and apply $f=q_Xp_Y$ to its tracks. Fix $y\in Y$
and put $\rho=\suppv(p_Y(y))$. Along the first part, $L_X$ changes the
vertex support by at most one edge. Applying $q_Y$ and then $p_Y$ changes
it by at most one further edge. Hence, after the final application of
$q_X$, the image lies in
$\bigcup_{a\in N_2(\rho)}U_{X,a}^R$.

Along the second part, property (2) gives
$\suppv(p_YH_Y(y,t))\subseteq N_1(\rho)$, so its image under $q_X$ lies
in $\bigcup_{a\in N_1(\rho)}U_{X,a}^R$. Thus the image under $f$ of the
entire concatenated track is contained in
$\bigcup_{a\in N_2(\rho)}U_{X,a}^R$, and hence in one member of $\alpha$.

Therefore $f$ is an $\alpha$-equivalence. The constant $6$ in
\eqref{eq:alpha-small} is  not sharp; the estimates in this
proposition use only two steps in the $1$-skeleton.
\end{proof}

\subsection{Fine towers from a PL structure}\label{subsec:fine-towers}

A PL manifold here means one admitting a combinatorial triangulation,
as specified after Theorem~\ref{thm:rigidity-intro}; the acting
homeomorphisms need not be PL. For a nonempty simplex $\sigma$ of $K$,
write $\St_K(\sigma)$ for its \emph{open star}: the union of the relative
interiors of the simplices containing $\sigma$, or equivalently the set
where $\lambda_a>0$ for every $a\in\sigma$. This is distinct from the
closed star $\overline{\St}_K(\sigma)$ used above.

\begin{lem}\label{lem:fine-towers}
Let $M$ be a closed connected PL manifold of positive dimension, and let
$\alpha$ be an open cover
of $M$. Then there are a finite combinatorial triangulation $K$ of $M$, an
integer $R\geq\dim K+2$, and a covering support tower of depth $R$ with
nerve $K$ which is $\alpha$-small. The triangulation $K$ and depth $R$
can be chosen independently of the locally approximating group acting
on $M$; only the support witnesses depend on that group.
\end{lem}

\begin{proof}
Fix a combinatorial triangulation of $M$ and a compatible metric. After
sufficiently many subdivisions, its mesh is arbitrarily small. We first
choose the subdivision finely enough to obtain the control required by
$\alpha$.

Let $\mesh(K)$ denote the mesh of the triangulation. If $a$ and $b$ have
graph distance at most $j$ in the $1$-skeleton, then they can be joined by
at most $j$ edges, each contained in a simplex of diameter at most
$\mesh(K)$. Moreover, every point of $\St_K(a)$ lies at distance at most
$\mesh(K)$ from $a$. It follows, for example, that for every nonempty
simplex $\sigma$,

\[
\diam\left(
 \bigcup_{a\in N_6(\sigma)}\St_K(a)
 \right)
 \leq 15\,\mesh(K).
\]

Choose the subdivision so that this quantity is smaller than a Lebesgue
number for $\alpha$. Then, for every nonempty $\sigma$, the union on the
left is contained in one member of $\alpha$. Fix this triangulation $K$,
put $n=\dim K=\dim M$, and choose $R\geq n+2$.

Let $(\lambda_a)_a$ be the barycentric coordinate functions of $K$. Choose
numbers

$0<\varepsilon_R^-<\varepsilon_R^+ <\varepsilon_{R-1}^-<\varepsilon_{R-1}^+ <\cdots<
\varepsilon_0^-<\varepsilon_0^+<\frac1{n+1}$.

For every vertex $a$ and $0\leq j\leq R$, put

\[
P_a^j=\{\lambda_a>\varepsilon_j^+\},
 \qquad
 Q_a^j=\{\lambda_a>\varepsilon_j^-\}.
\]

For every $j$, both families $(P_a^j)_a$ and $(Q_a^j)_a$ have nerve exactly
$K$. Moreover, $(P_a^0)_a$ covers $M$, since at every point some
barycentric coordinate is at least $1/(n+1)$.

We record explicitly that every nonempty finite intersection occurring here
is an open $n$-ball. Let $\sigma\in K$ be nonempty and let $0<\varepsilon<1/|\sigma|$.
On every simplex containing $\sigma$, the change of barycentric coordinates

\[
\lambda_a\longmapsto
 \begin{cases}
 \varepsilon+(1-|\sigma|\varepsilon)\lambda_a,
      &a\in\sigma,\\
 (1-|\sigma|\varepsilon)\lambda_a,
      &a\notin\sigma
 \end{cases}
\]

maps the set on which $\lambda_a>0$ for every $a\in\sigma$ onto the set on
which $\lambda_a>\varepsilon$ for every $a\in\sigma$. Its inverse subtracts
$\varepsilon$ from the coordinates indexed by $\sigma$ and divides all
coordinates by $1-|\sigma|\varepsilon$. These maps agree on common faces.

The combinatorial-triangulation hypothesis is used here: the
PL-spherical links imply that $\St_K(\sigma)$ is an open $n$-ball.
For a top-dimensional simplex, it is simply the relative interior of
that simplex. Hence every nonempty $P_\sigma^j$ and $Q_\sigma^j$ is an
open $n$-ball.

For every $a$ and $j$, apply Lemma~\ref{lem:large-subset} to choose
$U_a^j\in\DD$ such that
$\overline{P_a^j}\subseteq U_a^j$ and
$\overline{U_a^j}\subseteq Q_a^j$. Since
$P_a^j\subseteq U_a^j\subseteq Q_a^j$, the family $(U_a^j)_a$ also has
nerve exactly $K$. Since $(P_a^0)_a$ covers $M$, so does $(U_a^0)_a$.

Let $\sigma\in K$ and $j<R$. The choice of thresholds gives

\[
\overline{U_\sigma^j}
 \subseteq
 \bigcap_{a\in\sigma}\overline{Q_a^j}
 \subseteq
 P_\sigma^{j+1}
 \subseteq
 U_\sigma^{j+1}.
\]

The middle set is compact and lies in the open $n$-ball
$P_\sigma^{j+1}$. It is therefore contained in the interior of a collared
closed $n$-ball contained in $P_\sigma^{j+1}$. This gives the inclusions
required in~\eqref{eq:tower-inclusions}.

Finally, $U_a^R\subseteq Q_a^R\subseteq\St_K(a)$. Hence, for every nonempty
$\sigma$,

\[
\bigcup_{a\in N_6(\sigma)}U_a^R
 \subseteq
 \bigcup_{a\in N_6(\sigma)}\St_K(a),
\]

and the latter is contained in one member of $\alpha$ by the choice of
$K$. Thus the tower is $\alpha$-small.
\end{proof}

Proposition~\ref{prop:dimension} uses the local product version of this
construction. For fine towers on $M$, the PL hypothesis supplies the
open-ball intersections through the simplex-link conditions. An arbitrary
topological triangulation need not satisfy those conditions, so this
argument does not justify weakening combinatorial triangulability to
arbitrary triangulability. The nerve and controlled-comparison arguments
impose no triangulation assumption on the comparison manifold.

\subsection{From controlled homotopy to homeomorphism}
\label{subsec:alpha-approximation}

We use the following form of the controlled approximation theorem for closed
manifolds.

\begin{thm}[Controlled approximation]\label{thm:alpha-approximation}
Let $M$ be a fixed closed topological $n$-manifold, with $n\geq2$. There is
an open cover $\beta$ of $M$ such that, whenever $N$ is a closed topological
$n$-manifold and $f:N\to M$ is a $\beta$-equivalence, one has $N\cong M$.
More precisely, given an open cover $\alpha$ of $M$, the cover $\beta$ may
be chosen so that every such $f$ is $\alpha$-close to a homeomorphism
$N\to M$.
\end{thm}

For $n\geq5$, this is the Chapman--Ferry approximation theorem
\cite{ChapmanFerry}. In dimension four we use
\cite[Theorem~4]{FerryWeinberger}, where the approximation theorem is
extended to dimension four using controlled thin $h$-cobordisms. The
surface and three-dimensional versions are due to Jakobsche
\cite{Jakobsche1983,Jakobsche1988}, with the Poincar\'e theorem supplying
the three-dimensional input \cite{Perelman2003,MorganTian2007}. The
closed-manifold formulation used here, with the same target-side
control convention, is stated in \cite[Theorem~4.7]{LytchakNagano2022};
see also the definition on page~145 and the attribution preceding that
theorem on page~146. It has no orientability hypothesis. The proofs in
Section~\ref{subsec:low-dimensions} spell out the low-dimensional
applications, with a direct classification proof for surfaces.

\begin{thm}\label{thm:action-rigidity}
Let $M$ be a closed connected PL manifold of positive dimension. There is a
first-order group sentence $\Theta_M$ such that, for every locally
approximating $\HH\leq\Homeo(N)$ on a compact connected topological
manifold $N$,

\[
\HH\models\Theta_M
 \quad\Longleftrightarrow\quad
 N\cong M.
\]

\end{thm}

\begin{proof}
Put $n=\dim M$. If $n\geq2$, choose an open cover $\beta$ of $M$ as
in Theorem~\ref{thm:alpha-approximation}; if $n=1$, put $\beta=\{M\}$.
In the construction of Lemma~\ref{lem:fine-towers}, first choose a
sufficiently fine finite combinatorial triangulation $K$ of $M$ and a
fixed depth $R\geq n+2$. These choices, and the intermediate geometric
sets $P_a^j,Q_a^j$, are independent of the acting group. The
support-envelope construction then gives, for every locally
approximating group acting on $M$, a $\beta$-small covering support
tower with this same nerve $K$ and depth $R$. Fix one such tower
$\UU^M$ on $M$.

Let $A$ be the vertex set of $K$, and let
$\Phi_{K,R}((U_a^j)_{a,j})$ be the first-order condition that the
indicated supports form a covering tower with nerve $K$ and depth $R$.
It requires the nerve condition at every level, the inclusions
\eqref{eq:tower-inclusions}, and density of the level-zero union.
Define
\[
 \Theta_M:=\exists (U_a^j)_{a\in A,\,0\leq j\leq R}\quad
             \Phi_{K,R}((U_a^j)_{a,j}).
\]
Here the notation abbreviates quantification over $(R+1)|A|$ support
variables, and hence over a fixed finite tuple of group variables.
All conjunctions indexed by simplices or levels are finite. Thus
$\Theta_M$ is a parameter-free sentence in the group language.
The cover $\beta$ is used only in choosing the finite data and the
reference tower on $M$; it is not a parameter or a quantified object
in $\Theta_M$. In particular, the sentence does not require its
witnesses on another manifold to satisfy a smallness condition.

Suppose that a locally approximating $\HH\leq\Homeo(N)$ on a compact
connected topological manifold $N$ satisfies $\Theta_M$, and let
$\UU^N$ be a witnessing tower. First, $N$ is closed: density of the
level-zero union and finiteness of $A$ give
\[
 N=\bigcup_{a\in A}\overline{U_{N,a}^0},
\]
and each set on the right is contained in $\Int N$ by the tower
condition. Next, Proposition~\ref{prop:controlled-nerve} gives
$N\simeq|K|\cong M$. If $m=\dim N$, the highest degree of nonzero
mod-two homology of the closed connected manifold $N$ is $m$, whereas
that of $M$ is $n$. Hence $m=n$. This argument does not require an
orientation on either manifold.

If $n\geq2$, apply Proposition~\ref{prop:controlled-comparison} with
$X=M$, $Y=N$, and $\alpha=\beta$. Only the fixed tower $\UU^M$ is
required to be $\beta$-small. Writing $p_M,q_M$ and $p_N,q_N$ for the
maps supplied by Proposition~\ref{prop:controlled-nerve}, the
comparison map is
\[
 f=q_Mp_N:N\longrightarrow M,
 \qquad
 g=q_Np_M:M\longrightarrow N.
\]
The comparison proposition controls the tracks of $fg\simeq\id_M$
and the images under $f$ of the tracks of $gf\simeq\id_N$, all in the
fixed target $M$. Thus $f$ is a $\beta$-equivalence, and
Theorem~\ref{thm:alpha-approximation} gives $N\cong M$.
If $n=1$, the closed connected one-manifolds $M$ and $N$ are both
homeomorphic to $S^1$. The choice $R\geq n+2=3$ above ensures that
Proposition~\ref{prop:controlled-nerve} also applies in this case.

Conversely, suppose that $h:M\to N$ is a homeomorphism. Transport the
fixed triangulation and the intermediate geometric sets $P_a^j,Q_a^j$
to $N$. Now, apply the support-envelope construction in
Lemma~\ref{lem:fine-towers} to the action of $\HH$, choosing
$U_{N,a}^j\in\DD$ with
\[
 h(\overline{P_a^j})\subseteq U_{N,a}^j,
 \qquad
 \overline{U_{N,a}^j}\subseteq h(Q_a^j).
\]
The same intersection and containment argument gives a covering
support tower with exactly the fixed nerve $K$ and depth $R$.
No further subdivision is needed: the support witnesses are constructed
anew for $\HH$, rather than transported from a different acting group.
Therefore $\HH\models\Theta_M$.
\end{proof}

Theorem~\ref{thm:rigidity-intro} follows immediately by transferring
$\Theta_M$ under elementary equivalence. The conclusion is
$N\cong M$, not merely $N\simeq M$; the passage from controlled homotopy
equivalence to homeomorphism is supplied by
Theorem~\ref{thm:alpha-approximation}, rather than by any form of the Borel
conjecture.

The sentence $\Theta_M$ depends on one sufficiently fine finite
triangulation of $M$ and a fixed finite depth, but not on the acting group.
No uniform sentence simultaneously classifying all manifolds is asserted,
and no computable classification of manifolds is involved.

\subsection{Dimensions two and three}\label{subsec:low-dimensions}

We give the low-dimensional conclusions directly. The surface case uses
classification rather than controlled approximation. In dimension three,
we specify the control required by the approximation theorem.

First recall a consequence of the finite nerve construction. Suppose a
covering support tower of depth $R\geq\dim K+2$ on a compact connected
manifold $N$ has nerve $K$, where $K$ triangulates a closed connected
$n$-manifold $M$. Density and finiteness give
\[
 N=\bigcup_a\overline{U_{N,a}^0}.
\]
Every set on the right lies in $\Int N$, by the ball-containment condition
between levels zero and one. Thus $N$ is closed. Proposition~\ref{prop:controlled-nerve}
gives $N\simeq|K|\cong M$. For any closed connected $m$-manifold, its
highest nonzero mod-two homology group is in degree $m$ and is isomorphic
to $\Ftwo$ \cite[Theorem~3.26]{Hatcher}. Hence $\dim N=n$. In particular,
neither closedness nor the dimension of $N$ is assumed in advance, and no
orientability hypothesis is needed.

\begin{cor}\label{cor:closed-surfaces}
Let $M$ be a closed connected topological surface, and let
$\GG\leq\Homeo(M)$ be locally approximating. If $\HH\equiv\GG$ admits a
locally approximating action on a compact connected manifold $N$, then
$N\cong M$. No component-selection or orientability hypothesis is required.
\end{cor}

\begin{proof}
The surface $M$ admits a combinatorial triangulation \cite{Moise1977}.
Apply Lemma~\ref{lem:fine-towers} with the cover $\{M\}$ to choose a finite
nerve $K$ triangulating $M$ and a covering support tower of fixed depth
$R\geq4$. Existence of a covering support tower with these fixed data is
a first-order group sentence. It holds in $\GG$ and hence in $\HH$.
The preceding observation gives a closed surface $N$ with $N\simeq M$.

The integral homology groups determine the orientability and genus of a
closed connected surface. For an orientable surface $S_g$ of genus
$g\geq0$, and a nonorientable surface $T_k$ of genus $k\geq1$, they are
\[
\begin{aligned}
 H_2(S_g;\Z)&\cong\Z,& H_1(S_g;\Z)&\cong\Z^{2g},\\
 H_2(T_k;\Z)&=0,& H_1(T_k;\Z)&\cong\Z^{k-1}\oplus\Z/2.
\end{aligned}
\]
See \cite[Examples~2.36--2.37]{Hatcher}. Since $M$ and $N$ have isomorphic
integral homology, they have the same orientability and genus. The
classification of closed surfaces therefore gives $N\cong M$, including
the sphere and projective-plane cases. No approximation theorem is used
in this proof.
\end{proof}

For dimension three, we use the closed-manifold form of Jakobsche's
approximation theorem \cite{Jakobsche1988}, after the Poincar\'e theorem
\cite{Perelman2003,MorganTian2007}. A precise formulation is
\cite[Theorem~4.7]{LytchakNagano2022}; its convention for controlled
homotopy equivalences is given on page~145. With a compatible metric $d_M$
on a fixed closed topological three-manifold $M$, for every $\eta>0$
there is $\delta>0$ such that every $\delta$-equivalence $f:N\to M$ from
a closed topological three-manifold is $\eta$-close to a homeomorphism.
Here $f$ is a $\delta$-equivalence if it has a homotopy inverse $g$ and
homotopies $A:fg\simeq\id_M$ and $B:gf\simeq\id_N$ satisfying
\[
\begin{aligned}
 \diam_{d_M} A(\{x\}\times[0,1])&<\delta &&(x\in M),\\
 \diam_{d_M} f\bigl(B(\{y\}\times[0,1])\bigr)&<\delta &&(y\in N).
\end{aligned}
\]
Thus both kinds of control are measured in $M$. No orientability,
irreducibility, or triangulation hypothesis is imposed on $N$ in this
approximation statement.

\begin{cor}\label{cor:closed-three-manifolds}
Let $M$ be a closed connected topological three-manifold, and let
$\GG\leq\Homeo(M)$ be locally approximating. If $\HH\equiv\GG$ admits a
locally approximating action on a compact connected manifold $N$, then
$N\cong M$. No orientability hypothesis is required.
\end{cor}

\begin{proof}
Moise's theorem gives a combinatorial triangulation of $M$
\cite{Moise1952,Moise1977}. Fix $d_M$, choose $\eta>0$, and let $\delta>0$
be supplied by the approximation theorem just stated. Choose a finite
open cover $\beta$ of $M$ whose members have $d_M$-diameter less than
$\delta$. Lemma~\ref{lem:fine-towers} supplies a $\beta$-small covering
support tower on $M$, with finite nerve $K$ triangulating $M$ and fixed
depth $R\geq5$.

Let $\Psi_{K,R}$ assert the existence of a covering support tower with
this nerve and depth. The cover $\beta$ determines the geometric choice
of $K$ and the reference tower, but is not an argument or parameter of
$\Psi_{K,R}$. Since $\GG\models\Psi_{K,R}$ and $\HH\equiv\GG$, there is
a witnessing tower on $N$. The observation at the start of this
subsection proves that $N$ is closed, that $N\simeq M$, and that
$\dim N=3$.

Apply Proposition~\ref{prop:controlled-comparison} to the two towers,
using the $\beta$-small tower on $M$ as the target tower. The maps are
\[
 f=q_Mp_N:N\longrightarrow M,\qquad
 g=q_Np_M:M\longrightarrow N.
\]
Every track of $fg\simeq\id_M$, and the $f$-image of every track of
$gf\simeq\id_N$, lies in one member of $\beta$. These are containments
of entire tracks, not merely pointwise bounds in changing cover sets.
They imply precisely the two displayed diameter bounds. Thus $f$ is a
$\delta$-equivalence and is $\eta$-close to a homeomorphism $N\to M$.
In particular, $N\cong M$. No smallness assumption on the witnessing
tower on $N$ is required.
\end{proof}

The Poincar\'e input does not assert that homotopy-equivalent closed
three-manifolds are homeomorphic. It removes the homotopy-three-ball
obstruction: a compact contractible topological three-manifold $C$ with
$\partial C\cong S^2$ is a ball. Indeed, cap its boundary with $B^3$.
The resulting closed manifold is simply connected by van Kampen, hence
is $S^3$ by the Poincar\'e theorem. Boundary collars make the gluing
sphere bicollared. The generalized Schoenflies theorem
\cite[Theorem~5]{Brown1960} then identifies $C$ with $B^3$. The bicollar
is essential here; no assertion about arbitrary wild spheres is used.
The passage from $N\simeq M$ to $N\cong M$ in the three-dimensional
corollary is supplied by controlled approximation, not by Poincar\'e
alone. Neither low-dimensional proof uses arithmetic interpretation or
component selection.

\section*{Acknowledgements}
The authors thank S.-h.~Kim, Y.~Lodha, A.~Nies, and C.~Rosendal for helpful
discussions. We thank A.~Nies for clarifying the single-sentence requirement
in the definition of quasi-finite axiomatizability. The authors thank a referee for finding an error
in an earlier version of this paper. The arguments in this paper have been substantially revised and
checked with the help of OpenAI's ChatGPT, which also simplified and corrected some mathematical
statements.

\def\cprime{$'$} \def\soft#1{\leavevmode\setbox0=\hbox{h}\dimen7=\ht0\advance
  \dimen7 by-1ex\relax\if t#1\relax\rlap{\raise.6\dimen7
  \hbox{\kern.3ex\char'47}}#1\relax\else\if T#1\relax
  \rlap{\raise.5\dimen7\hbox{\kern1.3ex\char'47}}#1\relax \else\if
  d#1\relax\rlap{\raise.5\dimen7\hbox{\kern.9ex \char'47}}#1\relax\else\if
  D#1\relax\rlap{\raise.5\dimen7 \hbox{\kern1.4ex\char'47}}#1\relax\else\if
  l#1\relax \rlap{\raise.5\dimen7\hbox{\kern.4ex\char'47}}#1\relax \else\if
  L#1\relax\rlap{\raise.5\dimen7\hbox{\kern.7ex
  \char'47}}#1\relax\else\message{accent \string\soft \space #1 not
  defined!}#1\relax\fi\fi\fi\fi\fi\fi}
\providecommand{\bysame}{\leavevmode\hbox to3em{\hrulefill}\thinspace}
\providecommand{\MR}{\relax\ifhmode\unskip\space\fi MR }
\providecommand{\MRhref}[2]{%
  \href{http://www.ams.org/mathscinet-getitem?mr=#1}{#2}
}
\providecommand{\href}[2]{#2}


\begin{thebibliography}{99}

\bibitem{AM2009}
Tuna Altinel and Alexey Muranov, \emph{Interpr\'etation de l'arithm\'etique
  dans certains groupes de permutations affines par morceaux d'un intervalle},
  J. Inst. Math. Jussieu \textbf{8} (2009), no.~4, 623--652. \MR{2540875}

\bibitem{BarwiseSchlipf}
Jon Barwise and John Schlipf, \emph{An introduction to recursively saturated
and resplendent models}, J. Symbolic Logic \textbf{41} (1976), no.~2,
531--536.

\bibitem{Borsuk1948}
Karol Borsuk, \emph{On the imbedding of systems of compacta in simplicial
complexes}, Fund. Math. \textbf{35} (1948), 217--234.

\bibitem{Brown1960}
Morton Brown, \emph{A proof of the generalized Schoenflies theorem},
Bull. Amer. Math. Soc. \textbf{66} (1960), 74--76.

\bibitem{CFP1996}
J.~W. Cannon, W.~J. Floyd, and W.~R. Parry, \emph{Introductory notes on
  {R}ichard {T}hompson's groups}, Enseign. Math. (2) \textbf{42} (1996),
  no.~3--4, 215--256. \MR{1426438}

\bibitem{ChapmanFerry}
T.~A. Chapman and Steve Ferry, \emph{Approximating homotopy equivalences by
  homeomorphisms}, Amer. J. Math. \textbf{101} (1979), no.~3, 583--607.
  \MR{534907}

\bibitem{FerryWeinberger}
Steven~C. Ferry and Shmuel Weinberger, \emph{Curvature, tangentiality, and
  controlled topology}, Invent. Math. \textbf{105} (1991), no.~2, 401--414.
  \MR{1115548}

\bibitem{GalGis2017}
\'Swiatos\l aw~R. Gal and Jakub Gismatullin, \emph{Uniform simplicity of groups
  with proximal action}, Trans. Amer. Math. Soc. Ser. B \textbf{4} (2017),
  110--130. With an appendix by Nir Lazarovich. \MR{3693109}

\bibitem{GuelLiousse23}
Nancy Guelman and Isabelle Liousse, \emph{Uniform simplicity for subgroups of
  piecewise continuous bijections of the unit interval}, Bull. Lond. Math. Soc.
  \textbf{55} (2023), no.~5, 2341--2362. With an appendix by Pierre Arnoux.
  \MR{4672899}

\bibitem{Hatcher}
Allen Hatcher, \emph{Algebraic topology}, Cambridge University Press,
Cambridge, 2002.

\bibitem{Jakobsche1983}
W\l adys\l aw Jakobsche, \emph{Approximating homotopy equivalences of surfaces
by homeomorphisms}, Fund. Math. \textbf{118} (1983), no.~1, 1--9.

\bibitem{Jakobsche1988}
W\l adys\l aw Jakobsche, \emph{Approximating homotopy equivalences of
3-manifolds by homeomorphisms}, Fund. Math. \textbf{130} (1988), no.~3,
157--168.

\bibitem{khelif}
Anatole Khelif, \emph{Bi-interpr\'etabilit\'e{} et structures {QFA}: \'etude de
  groupes r\'esolubles et des anneaux commutatifs}, C. R. Math. Acad. Sci.
  Paris \textbf{345} (2007), no.~2, 59--61. \MR{2343552}

\bibitem{KK2018}
Sang-hyun Kim and Thomas Koberda, Free products and the algebraic structure of diffeomorphism groups. {\em J. Topol.}. \textbf{11}, 1054-1076 (2018), https://doi.org/10.1112/topo.12079

\bibitem{KK2020}
Sang-hyun Kim and Thomas Koberda, Diffeomorphism groups of critical regularity. {\em Invent. Math.}. \textbf{221}, 421-501 (2020), https://doi.org/10.1007/s00222-020-00953-y

\bibitem{KK2021book}
Sang-hyun Kim and Thomas Koberda, \emph{Structure and regularity of group
  actions on one-manifolds}, Springer Monographs in Mathematics, Springer,
  Cham, 2021. \MR{4381312}

\bibitem{KKL2019}
Kim, S., Koberda, T. \& Lodha, Y. Chain groups of homeomorphisms of the interval. {\em Ann. Sci. Éc. Norm. Supér. (4)}. \textbf{52}, 797-820 (2019), https://doi.org/10.24033/asens.2397

\bibitem{dlNKK22}
Sang-hyun Kim, Thomas Koberda, and J.~de~la Nuez~Gonz\'alez, \emph{First order
  rigidity of homeomorphism groups of manifolds}, Commun. Am. Math. Soc.
  \textbf{5} (2025), 144--194. \MR{4904300}

\bibitem{dlNKK25}
Sang-hyun Kim, Thomas Koberda, and J.~de~la Nuez~Gonz\'alez,
\emph{Elementary equivalence and diffeomorphism groups of smooth manifolds},
2025, \arxiv{2507.07427v3}

\bibitem{Kob-PP}
Koberda, T. Ping-pong lemmas with applications to geometry and topology. {\em Geometry, Topology And Dynamics Of Character Varieties}. \textbf{23} pp. 139-158 (2012)

\bibitem{KL2020}
Koberda, T. \& Lodha, Y. 2-chains and square roots of Thompson's group F. {\em Ergodic Theory Dynam. Systems}. \textbf{40}, 2515-2532 (2020), https://doi.org/10.1017/etds.2019.14

\bibitem{dlNK23}
Thomas Koberda and J.~de~la Nuez~Gonz\'alez, \emph{Uniform first order
  interpretation of the second order theory of countable groups of
  homeomorphisms}, Pacific J. Math. \textbf{341} (2026), 69--101.
  \arxiv{2312.16334}

\bibitem{lasserre}
Cl\'ement Lasserre, \emph{R. {J}. {T}hompson's groups {$F$} and {$T$} are
  bi-interpretable with the ring of the integers}, J. Symbolic Logic \textbf{79}
  (2014), no.~3, 693--711. \MR{3248780}

\bibitem{Leray1945}
Jean Leray, \emph{Sur la forme des espaces topologiques et sur les points
fixes des repr\'esentations}, J. Math. Pures Appl. (9) \textbf{24} (1945),
95--167.

\bibitem{LytchakNagano2022}
Alexander Lytchak and Koichi Nagano, \emph{Topological regularity of
spaces with an upper curvature bound}, J. Eur. Math. Soc.
\textbf{24} (2022), no.~1, 137--165. \doi{10.4171/JEMS/1091}

\bibitem{marker}
David Marker, \emph{Model theory: An introduction}, Graduate Texts in
Mathematics, vol.~217, Springer-Verlag, New York, 2002.

\bibitem{Moise1952}
Edwin E. Moise, \emph{Affine structures in 3-manifolds. V. The triangulation
theorem and Hauptvermutung}, Ann. of Math. (2) \textbf{56} (1952), 96--114.

\bibitem{Moise1977}
Edwin E. Moise, \emph{Geometric topology in dimensions 2 and 3}, Graduate
Texts in Mathematics, vol.~47, Springer-Verlag, New York--Heidelberg, 1977.

\bibitem{MorganTian2007}
John Morgan and Gang Tian, \emph{Ricci flow and the Poincar\'e conjecture},
Clay Mathematics Monographs, vol.~3, American Mathematical Society,
Providence, RI, and Clay Mathematics Institute, Cambridge, MA, 2007.

\bibitem{Nagorko2007}
Andrzej Nag\'orko, \emph{Carrier and nerve theorems in the extension theory},
Proc. Amer. Math. Soc. \textbf{135} (2007), no.~2, 551--558.

\bibitem{Nies-describing}
Andr\'e Nies, \emph{Describing groups}, Bull. Symbolic Logic \textbf{13}
  (2007), no.~3, 305--339. \MR{2359909}

\bibitem{Perelman2003}
Grisha Perelman, \emph{Finite extinction time for the solutions to the
Ricci flow on certain three-manifolds}, 2003, arXiv:math/0307245.

\bibitem{RourkeSanderson}
C.~P. Rourke and B.~J. Sanderson, \emph{Introduction to piecewise-linear
topology}, Springer Study Edition, Springer-Verlag, Berlin--Heidelberg,
1982.

\bibitem{Rubin1989}
Matatyahu Rubin, \emph{On the reconstruction of topological spaces from their
  groups of homeomorphisms}, Trans. Amer. Math. Soc. \textbf{312} (1989),
  no.~2, 487--538. \MR{988881}

\bibitem{Rubin1996}
\bysame, \emph{Locally moving groups and reconstruction problems}, Ordered
  groups and infinite permutation groups, Math. Appl., vol. 354, Kluwer Acad.
  Publ., Dordrecht, 1996, pp.~121--157. \MR{1486199}

\bibitem{Rubin2013announcement}
Matatyahu Rubin, \emph{The complete theory of every locally moving group is
undecidable}, Logic Seminar, Institute of Mathematics, Prague,
December 16, 2013, seminar announcement.\newline
\url{https://users.math.cas.cz/~thapen/seminar.html}.

\bibitem{Rubin2014announcement}
Matatyahu Rubin, \emph{Locally moving groups, reconstruction and
undecidability}, Kempner Colloquium, University of Colorado Boulder,
September 9, 2014, abstract, available at
\url{https://math.colorado.edu/~casa/seminars/colloquium/abstracts/140909abs.pdf}.

\bibitem{Schlipf1978}
John S. Schlipf, \emph{Toward model theory through recursive saturation},
J. Symbolic Logic \textbf{43} (1978), no.~2, 183--206.

\bibitem{tent-ziegler}
Katrin Tent and Martin Ziegler, \emph{A course in model theory}, Lecture
Notes in Logic, vol.~40, Association for Symbolic Logic and Cambridge
University Press, Cambridge, 2012.

\bibitem{Weil1952}
Andr\'e Weil, \emph{Sur les th\'eor\`emes de de Rham}, Comment. Math. Helv.
\textbf{26} (1952), 119--145.
\end{thebibliography}
\end{document}